\documentclass[11pt]{article}

\usepackage{amsmath,amssymb,amsfonts,textcomp,amsthm,xifthen,psfrag,graphicx,color}
\usepackage[T1]{fontenc}

\date{}

\newtheorem{theorem}{Theorem}
\newtheorem{lemma}[theorem]{Lemma}
\newtheorem{cor}[theorem]{Corollary}

\newtheorem{remark}[theorem]{Remark}
\theoremstyle{definition} 

\DeclareMathOperator*{\argmin}{arg\, min}
\DeclareMathOperator{\ext}{\mathcal{E}}

\DeclareMathOperator{\V}{\mathbb{V}}
\DeclareMathOperator{\W}{\mathbb{W}}

\newcommand{\ip}[2]{\langle#1\hspace*{.5mm},#2\rangle}
\newcommand{\dual}[2]{(#1\hspace*{.5mm},#2)}
\newcommand{\vdual}[2]{(#1\hspace*{.5mm},#2)}
\newcommand{\abs}[1]{\vert #1 \vert}
\newcommand{\norm}[3][]{#1\|#2#1\|_{#3}}

\newcommand{\diam}{\mathrm{diam}}

\newcommand{\wat}{\widehat}
\def\pwnabla{\nabla_\TT}

\def\div{{\rm div\,}}
\def\pwdiv{ {\rm div}_{\TT}\,}

\newcommand{\trace}{\gamma}

\newcommand{\hp}{{hp}}
\def\uc{u^c}
\def\Omc{\Omega^c}

\newcommand{\BEW}{{\rm hy}}
\newcommand{\BEV}{{\rm sl}}
\newcommand{\BEC}{{\rm ca}}
\newcommand{\LS}{{\rm LS}}

\newcommand{\R}{\ensuremath{\mathbb{R}}}

\newcommand{\HH}{\ensuremath{\mathbf{H}}}

\newcommand{\LL}{\ensuremath{\mathbf{L}}}

\newcommand{\nn}{\ensuremath{\mathbf{n}}}
\newcommand{\vv}{\ensuremath{\boldsymbol{v}}}
\newcommand{\ww}{\ensuremath{\boldsymbol{w}}}
\newcommand{\TT}{\ensuremath{\mathcal{T}}}

\newcommand{\cS}{\ensuremath{\mathcal{S}}}

\newcommand{\el}{\ensuremath{T}}

\newcommand{\OO}{\ensuremath{\mathcal{O}}}
\newcommand{\EE}{\ensuremath{\mathcal{E}}}

\newcommand{\ssigma}{{\boldsymbol\sigma}}
\newcommand{\ttau}{{\boldsymbol\tau}}
\newcommand{\qq}{{\boldsymbol{q}}}
\newcommand{\uu}{\boldsymbol{u}}

\newcommand{\slo}{\mathcal{V}}
\newcommand{\slp}{\widetilde\slo}
\newcommand{\hyp}{\mathcal{W}}
\newcommand{\dlo}{\mathcal{K}}
\newcommand{\dlp}{\widetilde\dlo}
\newcommand{\adlo}{\mathcal{K}'}

\newcounter{constantsnumber}
\def\setc#1{
  \ifthenelse{\equal{#1}{poinc}}{C_{\rm edge}}{ 
   \refstepcounter{constantsnumber}
   \label{const#1}C_{\theconstantsnumber}}}

\def\c#1{
  \ifthenelse{\equal{#1}{poinc}}{C_{\rm edge}}{ 
    C_{\ref{const#1}}}}

\newcommand{\Omat}{\mathbf{O}}
\newcommand{\Mmat}{\mathbf{M}}
\newcommand{\Pmat}{\mathbf{P}}
\newcommand{\xx}{\mathbf{x}}
\newcommand{\yy}{\mathbf{y}}
\newcommand{\est}{\operatorname{est}}

\title{On the coupling of DPG and BEM
\thanks{Supported by CONICYT through FONDECYT projects 1150056, 3140614, 3150012,
        and Anillo ACT1118 (ANANUM).}}
\author{
Thomas~F\"uhrer$^\dagger$
\and
Norbert Heuer$^\dagger$
\and
Michael Karkulik\thanks{
Facultad de Matem\'aticas, Pontificia Universidad Cat\'olica de Chile,
Avenida Vicu\~na Mackenna 4860, Santiago, Chile,
email: {\tt \{tofuhrer,nheuer,mkarkulik\}@mat.puc.cl}}}

\begin{document}
\maketitle
\begin{abstract}
  We develop and analyze strategies to couple the discontinuous Petrov-Galerkin method with
  optimal test functions to (i) least-squares boundary elements and (ii)
  various variants of standard Galerkin boundary elements.
  Essential feature of our methods is that, despite the use of boundary integral equations,
  optimal test functions have to be computed only locally.
  We apply our findings to a standard transmission problem in full space and
  present numerical experiments to validate our theory.

\bigskip
\noindent
{\em Key words}:
transmission problem, DPG method with optimal test functions, boundary elements,
least-squares method, coupling, ultra-weak formulation, Calder\'on projector

\noindent
{\em AMS Subject Classification}: 65N30, 35J20, 65N38
\end{abstract}

\section{Introduction}

In its current form, the discontinuous Petrov-Galerkin (DPG) method with optimal test functions
has been introduced by Demkowicz and Gopalakrishnan \cite{DemkowiczG_11_ADM,DemkowiczG_11_CDP}.
Principal objective is to guarantee uniform stability of the discrete scheme by the use of
specific \emph{optimal test functions}, in particular
for singularly perturbed \cite{DemkowiczH_13_RDM,ChanHBTD_14_RDM,BroersenS_14_RPG,BroersenS_15_PGD}
and wave problems \cite{ZitelliMDGPC_11_CDP,DemkowiczGMZ_12_WEA,GopalakrishnanMO_14_DDE}.

Essential feature of this method is that optimal test functions be calculated (approximatively)
locally on elements. This is well understood for partial differential equations on bounded domains.
When dealing with problems in unbounded domains (like transmission problems)
boundary integral equations are a natural way to deal with the exterior part. Their numerical analysis
usually requires the use of trace spaces, which are Sobolev spaces of non-integer orders.
As the integral operators themselves, corresponding norms are non-local.
That is, they cannot be represented equivalently as broken norms over elements and, therefore,
there are no straightforward efficient techniques for the calculation of optimal test functions
in these cases. Perhaps surprisingly, in \cite{HeuerP_14_UFH,HeuerK_DPG} we found an
ultra-weak variational formulation for hypersingular operators that is well posed in integer-order
Sobolev spaces as long as the underlying polygon (in two dimensions) or surface
(in three dimensions) is closed. Then, also the DPG framework
with optimal test functions goes through without complications (this is different for open
curves and surfaces). Still, the underlying formulation involves boundary integral operators.
This implies that optimal test functions cannot be calculated locally, and this seems natural for
global operators.

Now, returning to transmission problems with variational formulations comprising boundary integral
operators, application of DPG technology faces the problem of how to efficiently calculate optimal
test functions. In \cite{HeuerK_DPGb} we propose to globally apply this technique without specific
separation between differential and integral operators. It turns out that, as expected,
the integral operators generate global effects for optimal test functions, but only at a strip along
the interface (on which the integral operators live). In this way, one partially looses locality
in the calculation of optimal test functions.

In this paper, we pursue the idea of splitting approximations for partial differential operators from
those of boundary integral operators. This is very natural when coupling finite elements with
the boundary element method (BEM). Indeed, there is a long list of references and we only cite
\cite{JohnsonN_80_CBI,Costabel_88_SMC,CostabelS_88_CFE,BielakM_91_SFE} to name a few classical ones,
and refer to more recent developments \cite{Sayas_09_VJN,Steinbach_11_NSO,GaticaHS_12_RHB}
for detailed discussions.
However, when trying to use similar coupling techniques of DPG with BEM, there is a certain conflict
of frameworks. DPG-analysis is based upon arguments from functional analysis (even at the discrete level)
and the assignment of optimal test functions to any approximating basis function. This generates
square systems and complicates the coupling with Galerkin boundary elements where some of the approximating
finite element basis functions are tested with traditional basis functions from other discrete test spaces.

Noting that the DPG method with optimal test functions is in fact a least-squares method, it seems
most natural to couple it with a least-squares boundary element method. Previously, least-squares
boundary elements have been coupled with finite elements in
\cite{GaticaHS_03_LSC,MaischakS_04_LSC,MaischakOS_12_LSF}, and also in combination with the DPG method
they perform well, as we will show.
Additionally, we propose three coupling methods with Galerkin boundary elements,
one with hypersingular operator, one with weakly singular operator, and another one with
a combination of both.

Differently from traditional coupling methods, where some of the Cauchy data are represented directly
via boundary integral operators, our boundary element equations work as constraints for the DPG scheme.
In this way, ellipticity of the resulting bilinear form can be achieved and the well-posedness of
the discretization follows from that of the continuous formulation. We note, however, that
in the case of the Galerkin variants, proofs
of ellipticity require weighting of the DPG part with a positive number whose size is in principle
unknown. In our numerical experiments we always choose the constant one and observe perfect results.
Numerical tests for our model problems (and the specific domain under consideration)
indicate that only a small weighting of the order of $1/3$ or less
results in a stiffness matrix with symmetric part which is not positive definite.

Essential advantage of all our coupling methods is that optimal test functions have to be calculated only
for terms involving differential operators. In this way, locality of the corresponding
\emph{trial-to-test operator} (that maps ansatz functions to optimal test functions)
is fully maintained and DPG technology is applicable without interference of non-local operators.
In fact, we expect that in this way DPG formulations, that are specifically designed for singularly
perturbed and wave problems, can be coupled to boundary elements.
On the other hand, the analysis of two of our methods (least-squares BEM and hypersingular operator)
is based upon the fact that the kernel of the interior
formulation is strongly related to solutions of the exterior part of the problem.
In fact, Calder\'on boundary integral operators (used for our coupling schemes) map kernel functions
of the interior problem to components of their Cauchy data. This relation is maintained whenever
the interior and exterior differential operators of the transmission problem are identical.
At the moment it is unclear whether our analysis of those two methods can be extended to the case
of different interior and exterior problems.

To prove ellipticity of the coupling variant with weakly singular operator and Galerkin approximation,
we apply a recent technique from Sayas and Steinbach \cite{Sayas_09_VJN,Steinbach_11_NSO}. Here,
only gentle relations between Cauchy data of interior and exterior problems are needed, and
the proof extends to interior operators of the form $-\div A(x)\nabla u$ with minimum eigenvalue
of $A$ strictly bounded below by $1/4$.

The analysis of our fourth coupling scheme with a combination
of hypersingular and weakly singular operators, also does not rely on a specific relation of Cauchy
data. It is the coupling variant which is relatively straightforwardly applicable to more general problems
combining, e.g., singularly perturbed PDEs on bounded domains with linear, homogeneous PDEs
and constant coefficients in the exterior. Contrary to our analysis of the coupling with weakly singular
operator, such an extension would not need a condition on the minimum eigenvalue of the interior differential operator.

The remainder of this paper is as follows. In Section~\ref{sec_math} we present the model problem,
recall some Sobolev spaces and norms (Section~\ref{sec_Sob}),
revisit properties of boundary integral operators (Section~\ref{sec_BIO}),
present our three coupled schemes (Section~\ref{sec_schemes}),
and state their well-posedness and quasi-optimal convergence
(Theorems~\ref{thm_equiv},~\ref{thm_W},~\ref{thm_V}). The presentation of the schemes is first done
in an abstract way (in Section~\ref{sec_abstract}). We then recall an ultra-weak variational
formulation of the interior part of the problem (Section~\ref{sec_fem}) and deal with the
three coupling schemes, respectively, in Subsections~\ref{sec_LS},~\ref{sec_W}, and~\ref{sec_V}.
Further technical details and proofs of the main theorems are given in Section~\ref{sec_LS_tech}.
In Section~\ref{sec_num} we present several numerical examples. We also give details on the
implementation of our least-squares boundary elements (Section~\ref{sec_LS_impl}) and present
a simple a posteriori error estimator (Section~\ref{sec_est}) that serves as a common
error bound for comparison.

\section{Mathematical setting and main results} \label{sec_math}
Let $\Omega\subset\R^d$, $d\in\{2,3\}$, be a bounded, simply connected Lipschitz domain
with boundary $\Gamma$, $\Omc:=\R^d\setminus\overline\Omega$, and
normal vector $\nn_\Omega$ on $\Gamma$ pointing in direction of $\Omc$.
We consider the following model transmission problem:
given $f\in L_2(\Omega)$, $u_0\in H^{1/2}(\Gamma)$, $\phi_0\in H^{-1/2}(\Gamma)$,
find $u\in H^1(\Omega)$ and $u^c\in H^1_{\mathrm{loc}}(\Omc)$ such that
\begin{subequations}\label{tp}
\begin{align}
  -\Delta u &= f \text{ in } \Omega,\label{tpa}\\
  \Delta \uc &= 0 \text{ in } \Omc, \label{tpb}\\
  u-\uc &= u_0 \text{ on } \Gamma,\label{tpc}\\
  \frac{\partial}{\partial\nn_\Omega}(u-\uc) &= \phi_0 \text{ on } \Gamma,\label{tpd}\\
  \uc(x) &= \OO(\abs{x}^{-1}) \text{ as } \abs{x}\rightarrow\infty.\label{tpe}
\end{align}
\end{subequations}
Here, $H^1(\Omega)$, $H^1_{\mathrm{loc}}(\Omc)$ denote standard Sobolev spaces with trace space
$H^{1/2}(\Gamma)$ and its dual $H^{-1/2}(\Gamma)$, see Section~\ref{sec_Sob} for precise definitions.
For $d=2$, we assume in addition that $\int_\Omega f + \int_\Gamma \phi_0 = 0$ 
(which ensures the correct behavior~\eqref{tpe} at infinity).

\subsection{Abstract coupling framework} \label{sec_abstract}
Our aim is to use a DPG formulation for the interior part of problem \eqref{tp}, and to couple
it with boundary integral equation(s) to take the exterior part into account. We will propose
four coupled schemes whose settings are formally identical.
In this section we present the corresponding abstract framework.

Let $U$ and $V$ be, respectively, a reflexive Banach space with norm $\norm{\cdot}{U}$
and a Hilbert space with inner product $\ip{\cdot}{\cdot}_V$ and norm $\norm{\cdot}{V}$.
For a given bilinear form $b(\cdot,\cdot):\;U\times V\to\R$ and linear functional
$L_V\in V'$ we consider the variational formulation
\begin{equation} \label{VF}
   \uu\in U:\quad b(\uu,\vv) = L_V(\vv)\quad\forall \vv\in V.
\end{equation}
We will also make use of the operator $B:\;U\to V'$ induced by the bilinear form $b$.
In our case, problem \eqref{VF} itself has no unique solution and will be coupled with a variational relation
\begin{equation} \label{Cauchy}
   \uu\in U:\quad c(\uu,\ww) = L_U(\ww)\quad\forall\ww\in U.
\end{equation}
Here, $c(\cdot,\cdot):\;U\times U\to\R$ is another bilinear form and $L_U\in U'$ a linear functional.
Essential feature in the discretization of \eqref{VF} is the use of so-called \emph{optimal test
functions}. They are generated through the \emph{weighted trial-to-test operator}
$\Theta_\beta:U\to V$ defined by
\begin{align} \label{ttt}
  \ip{\Theta_\beta\uu}{\vv}_V = \beta\, b(\uu,\vv)\quad\forall\vv\in V\quad (\beta>0).
\end{align}
Using this operator for a constant $\beta>0$, we combine the relations \eqref{VF} and \eqref{Cauchy}
to the following variational formulation:
\begin{align}\label{eq_dpgcoupling}
  \uu\in U:\quad b(\uu,\Theta_\beta\ww) + c(\uu,\ww) = L_V(\Theta_\beta\ww) + L_U(\ww)
  \quad\forall\ww\in U.
\end{align}
Note that in this DPG-type formulation, the trial-to-test operator makes use of the bilinear form
$b(\cdot,\cdot)$ exclusively. In our specific schemes introduced below, this allows for maintaining
locality of the application of $\Theta_\beta$ while the bilinear form $c(\cdot,\cdot)$ comprises
non-local (boundary integral) operators.

For a finite-dimensional approximation space $U_\hp\subset U$, the discrete scheme then is
\begin{align}\label{eq_dpgcoupling_discrete}
  \uu_\hp\in U_\hp:\quad b(\uu_\hp,\Theta_\beta\ww) + c(\uu_\hp,\ww) = L_V(\Theta_\beta\ww) + L_U(\ww)
  \quad\forall\ww\in U_\hp.
\end{align} 
The following lemma states well-posedness of~\eqref{eq_dpgcoupling}
and best-approximation property of its discrete form~\eqref{eq_dpgcoupling_discrete}.
It amounts to using the Lax-Milgram lemma (in the case of Hilbert spaces)
or Babu\v{s}ka-Brezzi theory (in the case of Banach spaces)
and C\'ea's lemma for elliptic variational formulations.
Of course, this is a standard procedure in least-squares analysis, see, e.g., \cite{BochevG_09_LSF}.

\begin{lemma}\label{lem_dpgcoupling}
Suppose that $b:\;U\times V\to\R$ and $c:\;U\times U\to\R$ are bounded bilinear forms with
bounds $C_b$ and $C_c$, respectively. Furthermore, assume that, for given $\beta>0$,
there is $\alpha>0$ such that
\begin{align*}
   \alpha \norm{\ww}{U}^2 \le b(\ww,\Theta_\beta\ww) + c(\ww,\ww)\quad\forall\ww\in U,
\end{align*}
that is, the bilinear form on the left-hand side of~\eqref{eq_dpgcoupling} is $U$-elliptic.
Then, there is a unique solution $\uu$ of~\eqref{eq_dpgcoupling} with

\begin{align*} 
   \norm{\uu}{U}
   \le
   \frac 1\alpha \bigl(\beta C_b \norm{L_V}{V'} + \norm{L_U}{U'}\bigr).
\end{align*}
In addition, there is a unique solution $\uu_\hp$ of~\eqref{eq_dpgcoupling_discrete}, and it satisfies
\begin{align*}
   \norm{\uu-\uu_\hp}{U}
   \le
   \frac{\beta C_b^2+C_c}{\alpha} \inf_{\ww\in U_\hp} \norm{\uu-\ww}{U}.
\end{align*}
\end{lemma}

\begin{proof}
The statements are immediate consequences of the Lax-Milgram lemma (or Babu\v{s}ka-Brezzi theory)
and C\'ea's lemma. In fact, ellipticity of the bilinear form holds by assumption and boundedness of the
bilinear form and linear functionals follows by bounding
$\norm{\Theta_\beta\ww}{V} \le \beta C_b\norm{\ww}{U}$ for all $\ww\in U$.
To this end, note that $\Theta_\beta = \beta J_V^{-1}B$ with isometric Riesz map $J_V:\;V\to V'$.
To show the best approximation property we also make use of the orthogonality
$b(\uu-\uu_\hp,\Theta_\beta\ww)+c(\uu-\uu_\hp,\ww)=0$ $\forall\ww\in U_\hp$.
\end{proof}

In Section~\ref{sec_fem} we will specify~\eqref{VF} as an ultra-weak variational
formulation of the interior part of \eqref{tp},
and in Sections~\ref{sec_LS},~\ref{sec_W},~\ref{sec_V} we will present three possibilities
of selecting \eqref{Cauchy} as boundary integral equations.
Before doing so we introduce some Sobolev spaces and norms in Section~\ref{sec_Sob},
and recall properties of boundary integral operators in Section~\ref{sec_BIO}.

\subsection{Sobolev spaces and norms}\label{sec_Sob}
For a Lipschitz domain $\omega\subset\R^d$ we use the standard Sobolev spaces
$L_2(\omega)$, $H^1(\omega)$, $H^1_0(\omega)$ (the index $0$ denotes vanishing trace on the
boundary of $\omega$), $\HH(\div,\omega)$, $\HH_0(\div,\omega)$.
Vector-valued spaces and functions will be denoted by bold symbols.
Denoting by $\trace_\omega$ the trace operator acting on $H^1(\omega)$, we define trace spaces
\begin{align*}
  H^{1/2}(\partial\omega) := \left\{ \trace_\omega u;\; u\in H^1(\omega) \right\}
  \quad\text{ and its dual }\quad H^{-1/2}(\partial\omega) := \bigl(H^{1/2}(\partial\omega)\bigr)'
\end{align*}
and use the canonical norms.
Here, duality is understood with respect to $L_2(\partial\omega)$ as a pivot space,
i.e., using the extended $L_2(\partial\omega)$ inner product $\dual{\cdot}{\cdot}_{\partial\omega}$.
The $L_2(\Omega)$ inner product will be denoted by $\dual{\cdot}{\cdot}_\Omega$.
Let $\TT$ denote a disjoint partition of $\Omega$ into open Lipschitz sets $\el\in\TT$,
i.e., $\cup_{\el\in\TT}\overline\el = \overline\Omega$.
The set of all boundaries of all elements forms the skeleton
$\cS := \left\{ \partial\el \mid \el\in\TT \right\}$.
By $\nn_M$ we mean the outer normal vector on $\partial M$ for a Lipschitz set $M$.
On a partition $\TT$ we use product spaces $H^1(\TT)$ and $\HH(\div,\TT)$ and equip them
with corresponding product norms.
The symbols $\pwnabla$ and $\pwdiv$ denote, respectively, the $\TT$-piecewise gradient
and divergence operators. On the skeleton $\cS$ of $\TT$ we introduce the trace spaces
\begin{align*}
  H^{1/2}(\cS) &:=
  \Big\{ \wat u \in \Pi_{\el\in\TT}H^{1/2}(\partial\el);\;
         \exists w\in H^1(\Omega) \text{ such that } 
         \wat u|_{\partial\el} = w|_{\partial\el}\; \forall \el\in\TT \Big\},\\
  H^{-1/2}(\cS) &:=
  \Big\{ \wat\sigma \in \Pi_{\el\in\TT}H^{-1/2}(\partial\el);\;
         \exists \qq\in\HH(\div,\Omega) \text{ such that } 
         \wat\sigma|_{\partial\el} = (\qq\cdot\nn_{\el})|_{\partial\el}\; \forall\el\in\TT \Big\}.
\end{align*}
These spaces are equipped with the norms
\begin{subequations} \label{Hpm}
\begin{align}
  \norm{\wat u}{H^{1/2}(\cS)} &:=
  \inf \left\{ \norm{w}{H^1(\Omega)};\; w\in H^1(\Omega) \text{ such that }
               \wat u|_{\partial\el}=w|_{\partial\el}\; \forall\el\in\TT \right\},\\
  \norm{\wat\sigma}{H^{-1/2}(\cS)} &:=
  \inf \left\{ \norm{\qq}{\HH(\div,\Omega)};\; \qq\in\HH(\div,\Omega) \text{ such that }
               \wat\sigma|_{\partial\el}=(\qq\cdot\nn_{\el})|_{\partial\el}\; \forall\el\in\TT \right\}.
\end{align}
\end{subequations}
For our analysis we will also need the space
\begin{align*}
  H^{1/2}_{00}(\cS) &:=
  \bigl\{ \wat u\in H^{1/2}(\cS);\; \wat u|_\Gamma = 0 \bigr\}.
\end{align*}
For functions $\wat u\in H^{1/2}(\cS)$, $\wat\sigma\in H^{-1/2}(\cS)$
(which are elements of product spaces with components
$\wat u|_{\partial\el}\in H^{1/2}(\partial\el)$ and
$\wat\sigma|_{\partial\el}\in H^{-1/2}(\partial\el)$, respectively, for $\el\in\TT$),
and $\ttau\in\HH(\div,\TT)$, $v\in H^1(\TT)$ we use the notation
\begin{align*}
  \dual{\wat u}{\ttau\cdot\nn}_\cS
  := \sum_{\el\in\TT}\dual{\wat u|_{\partial\el}}{\ttau\cdot\nn_\el}_{\partial\el},\quad
  \dual{\wat\sigma}{v}_\cS
  := \sum_{\el\in\TT}\dual{\wat\sigma|_{\partial\el}}{v}_{\partial\el}.
\end{align*}
Furthermore, throughout the paper, suprema are taken over sets excluding the null element,
and the notation $A\lesssim B$ is used to say that $A\leq C\cdot B$ with a constant $C>0$ which
does not depend on any quantities of interest.
Correspondingly, the notation $A\gtrsim B$ is used,
and $A\simeq B$ means that $A\lesssim B$ and $B\lesssim A$.

\subsection{Boundary integral operators} \label{sec_BIO}
The exterior part of problem \eqref{tp} will be dealt with by boundary integral operators.
To this end we need some further definitions. The fundamental solution of the Laplacian is
\begin{align*}
  G(z) :=
    \begin{cases}
    -\frac{1}{2\pi}\log\abs{z}\quad& (d=2),\\
    \frac{1}{4\pi}\frac{1}{\abs{z}}& (d=3),
  \end{cases}
\end{align*}
and the corresponding single layer and double layer potentials are
\begin{align*}
  \slp\phi(x) := \int_\Gamma G(x-y)\phi(y)\,ds_y,\quad
  \dlp v(x) := \int_\Gamma \partial_{\nn_\Omega(y)}G(x-y)v(y)\,ds_y,
  \quad x\in \R^d\setminus\Gamma.
\end{align*}
Application of the trace operator $\trace_\Omega$ and the normal derivative
$\partial_{\nn_\Omega}$ gives rise to the three boundary integral operators
\begin{align}\label{slo_dlo}
  \begin{split}
  \slo := \trace_\Omega \slp,\qquad
  \dlo := 1/2 + \trace_\Omega\dlp,\qquad
  \hyp := -\partial_{\nn_\Omega} \dlp.
  \end{split}
\end{align}
They are the \textit{single layer}, \textit{double layer},
and \textit{hypersingular operators}, respectively. The adjoint operator of $\dlo$ is denoted by $\adlo$.
These operators are linear and bounded as mappings
$\slo: H^{-1/2}(\Gamma)\rightarrow H^{1/2}(\Gamma)$,
$\dlo: H^{1/2}(\Gamma)\rightarrow H^{1/2}(\Gamma)$,
$\adlo:H^{-1/2}(\Gamma)\rightarrow H^{-1/2}(\Gamma)$, and
$\hyp: H^{1/2}(\Gamma)\rightarrow H^{-1/2}(\Gamma)$.
We note that here holds $\ker(\hyp)=\mathrm{span}\{1\}$.
For $(u,\phi)\in H^{1/2}(\Gamma)\times H^{-1/2}(\Gamma)$ let us define the operators
\begin{align*}
  \V(u,\phi) &:= \slo\phi + (1/2-\dlo)u,\qquad
  \W(u,\phi) := \hyp u + (1/2+\adlo)\phi.
\end{align*}
By the boundedness of $\slo$, $\dlo$, $\adlo$, and $\hyp$, it follows that
\begin{align}\label{op_bound}
  \begin{split}
    \V&:\; H^{1/2}(\Gamma)\times H^{-1/2}(\Gamma)\to H^{1/2}(\Gamma),\\
    \W&:\; H^{1/2}(\Gamma)\times H^{-1/2}(\Gamma)\to H^{-1/2}(\Gamma)
  \end{split}
\end{align}
are bounded.
Furthermore, for the exterior Cauchy data of the harmonic function $u^c\in H^1_{\mathrm{loc}}(\Omc)$
(see~\eqref{tpb}) we have the so-called Calder\'on system
\begin{subequations}\label{BIE}
\begin{align}\label{BIEa}
  \V(u^c|_\Gamma, \partial_{\nn_\Omega}u^c) = 0,\\
  \W(u^c|_\Gamma, \partial_{\nn_\Omega}u^c) = 0.\label{BIEb}
\end{align}
\end{subequations}
For details and proofs we refer to classical references, e.g.
\cite{Costabel_88_BIO,McLean_00_SES,HsiaoW_08_BIE}.

\subsection{Ultra-weak finite element part}\label{sec_fem}
In this section we recall an ultra-weak variational formulation of the
interior part of the transmission problem \eqref{tp}. This is taken from \cite{DemkowiczG_11_ADM}
and corresponds to problem \eqref{VF} of the abstract framework in Section~\ref{sec_abstract}.

To this end, let $\TT$ be a partition of $\Omega$ with skeleton $\cS$.
Corresponding to $\TT$ we define fractional-order spaces $H^{1/2}(\cS)$ and
$H^{-1/2}(\cS)$ as in Section~\ref{sec_Sob}, and select
\[
   U := L_2(\Omega)\times\LL_2(\Omega)\times H^{1/2}(\cS)\times H^{-1/2}(\cS),\quad
   V := H^1(\TT)\times\HH(\div,\TT).
\]
The broken test space $V$ is provided with the canonical inner product denoted (as previously)
by $\ip{\cdot}{\cdot}_V$. Furthermore, we define the bilinear and linear forms $b$ and $L_V$ by
\begin{align} \label{ultra_bil_L}
   b(\uu,\vv)
   &:=
   \vdual{u}{\pwdiv\ttau}_\Omega + \vdual{\ssigma}{\pwnabla v+\ttau}_\Omega -
   \dual{\wat u}{\ttau\cdot\nn}_\cS - \dual{\wat\sigma}{v}_\cS,\quad
   L_V(\vv) := \vdual{f}{v}_\Omega
\end{align}
for $\uu=(u,\ssigma,\wat u,\wat\sigma)\in U$, $\vv=(v,\ttau)\in V$.

With this setting, \eqref{VF} is our ultra-weak variational formulation of the interior part
\eqref{tpa} of the transmission problem. For reference, we explicitly specify its strong form
\begin{align} \label{BuL}
   \uu:=(u,\ssigma,\wat u,\wat\sigma)\in U:\quad B\uu = L_V.
\end{align}
The weighted trial-to-test operator $\Theta_\beta$ is exactly as in \eqref{ttt}.

Contrary to~\cite{DemkowiczG_11_ADM} (where the model problem considers a homogeneous Dirichlet
boundary condition), our transmission problem uses a space $U$ whose third component
$H^{1/2}(\cS)$ does not incorporate such a condition on $\Gamma$.
We therefore anticipate that the operator $B:\;U\to V'$ has a non-trivial kernel whereas it is known
that
\[
   B:\;U_0:=L_2(\Omega)\times\LL_2(\Omega)\times H^{1/2}_{00}(\cS)\times H^{-1/2}(\cS)\ \to\ V'
\]
is invertible:

\begin{lemma}\label{lem_dg11}
  The bilinear form $b:\; U\times V\rightarrow\R$ is bounded with a constant
  independent of $\TT$. Furthermore, there holds
  \begin{align*}
    \sup_{\vv\in V} \frac{b(\uu_0,\vv)}{\norm{\vv}{V}}
    &\gtrsim \norm{\uu_0}{U} \quad\text{ for all }\uu_0 \in U_0, \text{ and}\\
    \sup_{\uu_0\in U_0} \frac{b(\uu_0,\vv)}{\norm{\uu_0}{U}}
    &> 0 \quad\text{ for all }\vv \in V\setminus\{0\}.
  \end{align*}
\end{lemma}

\begin{proof}
  The boundedness of $b$ is immediate by the definition of the norms.
  The remaining results follow from~\cite{DemkowiczG_11_ADM}. More specifically,
  by \cite[Theorem~4.2]{DemkowiczG_11_ADM} there holds the converse inf-sup condition
  $\norm{B\vv}{U_0'}\gtrsim\norm{\vv}{V}$ $\forall \vv\in V$,
  and $b$ is non-degenerate in the
  $U_0$-component by~\cite[Lemma~4.1]{DemkowiczG_11_ADM}.
  The statement then follows by the Babu\v{s}ka-Brezzi theory.
\end{proof}

Obviously, the non-uniqueness of \eqref{BuL} is due to missing boundary conditions. They
will be incorporated by Cauchy-data relations that stem from the transmission conditions
\eqref{tpc}, \eqref{tpd} and the exterior problem \eqref{tpb}, \eqref{tpe}.
For the handling of these data we need the following restriction operators.
\begin{align*}
  \begin{aligned}
    \gamma&:\;U\to H^{1/2}(\Gamma)\times H^{-1/2}(\Gamma),&\quad&
    \gamma(u,\ssigma,\wat u,\wat\sigma):=(\wat u|_\Gamma,\wat\sigma|_\Gamma)
    &\qquad&\text{(Cauchy data)},\\
    \gamma_0&:\;U\to H^{1/2}(\Gamma),&\quad&
    \gamma_0(u,\ssigma,\wat u,\wat\sigma):=\wat u|_\Gamma
    &\qquad&\text{(Dirichlet trace),} \\
    \gamma_\nn&:\;U\to H^{-1/2}(\Gamma),&\quad&
    \gamma_\nn(u,\ssigma,\wat u,\wat\sigma):=\wat \sigma|_\Gamma
    &\qquad&\text{(Neumann trace)}.
  \end{aligned}
\end{align*}
The boundedness of these operators is immediate.
\begin{lemma} \label{la_restriction}
  The operators $\gamma$, $\gamma_\nn$ and $\gamma_0$ are bounded.
\end{lemma}
\begin{proof}
By definition of $H^{1/2}(\Gamma)$ and its norm,
$\|\wat u\|_{H^{1/2}(\Gamma)}\le \|u\|_{H^1(\Omega)}$ for any $u\in H^1(\Omega)$ with $u=\wat u$
on $\Gamma$. That is, $\|\wat u\|_{H^{1/2}(\Gamma)}\le \|\wat u\|_{H^{1/2}(\cS)}$ by the
definition of the latter norm. Analogously, making use of the boundedness of the normal
component trace from $\HH(\div,\Omega)$ to $H^{-1/2}(\Gamma)$, we obtain
\[
   \|\wat\sigma\|_{H^{-1/2}(\Gamma)}
   \lesssim
   \inf\{\|\ssigma\|_{\HH(\div,\Omega)};\; \ssigma\cdot\nn_\Omega = \wat\sigma\ \text{on}\ \Gamma\}
   \le
   \|\wat\sigma\|_{H^{-1/2}(\cS)}.
\]
\end{proof}

\subsection{Four coupling methods} \label{sec_schemes}
In Section~\ref{sec_fem} we have fixed the interior part of the variational formulation
for the transmission problem \eqref{tp}. This corresponds to the abstract form \eqref{VF}.
It remains to add an equation \eqref{Cauchy} to incorporate the exterior part of the problem
and transmission conditions. In the next subsection we consider a boundary integral
equation of least-squares type for \eqref{Cauchy},
and in Subsections~\ref{sec_W},~\ref{sec_V},~\ref{sec_C} we study three Galerkin variants, with
hypersingular operator, single layer operator, and both operators, respectively (all with additional
rank-one terms).

\subsubsection{Coupling with least-squares boundary elements} \label{sec_LS}
We couple equation~\eqref{BuL} with the exterior problem by using the boundary integral equation
\eqref{BIEa} and jump relations~\eqref{tpc},~\eqref{tpd}.
With the operator $\V$ this yields the relation
\begin{align} \label{op}
   \V(\gamma\uu) = \V(u_0,\phi_0) \quad\text{in}\ H^{1/2}(\Gamma).
\end{align}
Denoting the inner product in $H^{1/2}(\Gamma)$ by $\ip{\cdot}{\cdot}_{H^{1/2}(\Gamma)}$,
a least-squares formulation of \eqref{op} is
\begin{align*}
  \ip{\V(\gamma\uu)}{\V(\gamma\ww)}_{H^{1/2}(\Gamma)}
  &= \ip{\V(u_0,\phi_0)}{\V(\gamma\ww)}_{H^{1/2}(\Gamma)}
  \quad\forall\ww\in U.
\end{align*}
According to the abstract framework from Section~\ref{sec_abstract}, we define
\begin{align*}
  c(\uu,\ww) &:= c_{\LS}(\uu,\ww) := \ip{\V(\gamma\uu)}{\V(\gamma\ww)}_{H^{1/2}(\Gamma)},\\
  L_U(\ww) &:= L_\LS(\ww) := \ip{\V(u_0,\phi_0)}{\V(\gamma\ww)}_{H^{1/2}(\Gamma)}.
\end{align*}
Then our DPG formulation coupled with a least-squares boundary integral equation is
\eqref{eq_dpgcoupling}:
\begin{align}\label{eq:lsc} 
  \uu\in U:\quad
  b(\uu,\Theta_\beta\ww) + c_\LS(\uu,\ww)
  = L_V(\Theta_\beta\ww) + L_\LS(\ww) \quad\forall \ww\in U.
\end{align}
The bilinear form $b$ and functional $L_V$ have been defined in Section~\ref{sec_fem}.
The coupled scheme \eqref{eq:lsc} is equivalent to the least-squares formulation
\begin{align} \label{LS}
   \uu=(u,\ssigma,\wat u,\wat\sigma)\in U:\quad
   \uu = \argmin_{\ww\in U} J_\beta(\ww; f,u_0,\phi_0)
\end{align}
with
\[
   J_\beta(\uu;f,u_0,\phi_0)
   :=
   \beta \|B\uu-L_V\|_{V'}^2 + \|\V(\gamma\uu) - \V(u_0,\phi_0)\|_{H^{1/2}(\Gamma)}^2.
\]
The corresponding discretization is \eqref{eq:lsc} with $U$ replaced by a finite-dimensional subspace
$U_\hp$, cf.~\eqref{eq_dpgcoupling_discrete}, or equivalently
\begin{align} \label{LS_discrete}
   \uu_\hp\in U_\hp:\quad
   \uu_\hp = \argmin_{\ww\in U_\hp} J_\beta(\ww; f,u_0,\phi_0).
\end{align}
Principal advantage of this scheme over the following ones is that, for scaling
parameter $\beta=1$, the resulting bilinear form is known to be elliptic so that the formulation
is well posed and its discretization converges quasi-optimally.

\begin{theorem} \label{thm_equiv}
  Let $\beta=1$. Problems~\eqref{tp} and~\eqref{LS} are equivalent and uniquely solvable.
  More precisely, let $(u,u^c)$ be the solution to \eqref{tp} and define
  $\uu:=(u,\ssigma,\wat u,\wat\sigma)$ with $\ssigma:=\nabla u$ in $\Omega$,
  $\wat u:=u|_\cS$ (element-wise), and $\wat\sigma:=\ssigma\cdot\nn_\el$ on $\partial\el$
  for any $\el\in\TT$. Then, $\uu$ satisfies \eqref{LS}.
  On the other hand, if $\uu=(u,\ssigma,\wat u,\wat\sigma)$ solves \eqref{LS},
  then $(u,u^c)$ with $u^c:=\dlp(\wat u|_\Gamma-u_0) - \slp(\wat\sigma|_\Gamma - \phi_0)$
  solves \eqref{tp}.

  Furthermore, for a finite-dimensional subspace $U_\hp\subset U$, there is a unique 
  solution $\uu_\hp\in U_\hp$ of \eqref{LS_discrete}, and there holds the best approximation
  property
  \begin{align*}
    \norm{\uu-\uu_\hp}{U} \lesssim
    \inf_{\ww\in U_\hp} \norm{\uu-\ww}{U}
  \end{align*}
  with hidden constant independent of $\TT$ and the polynomial degrees in $U_\hp$.
\end{theorem}

A proof of this theorem will be given in Section~\ref{sec_pf_equiv}.

\begin{remark} \label{rem2}
(i) In contrast to the method proposed in \cite{HeuerK_DPGb}, the trial-to-test operator
$\Theta:=\Theta_1$ used here can be implemented without considering any boundary integral operator.
Therefore, having defined the space $V$ as a product space of broken Sobolev spaces of integer orders,
the action of $\Theta$ is completely local and can be approximated in a standard way known from
DPG methods, cf.~\cite{GopalakrishnanQ_14_APD}.

(ii) The only difference of our trial-to-test operator with that of standard DPG schemes is that,
in our case, the bilinear form on the right-hand side defining $\Theta$ has a kernel. This does not
affect the implementation. Indeed, if there is $\ww\in U_{hp}\cap\ker B$ then $\Theta\ww=0$ and
equation~\eqref{eq:lsc} reduces to
\[
   \ip{\V(\gamma\uu_{hp})}{\gamma_0\ww}_{H^{1/2}(\Gamma)}
   =
   \ip{\V(u_0,\phi_0)}{\gamma_0\ww}_{H^{1/2}(\Gamma)},
\]
cf.~Lemmas~\ref{la_kerB} and~\ref{la_op}.

(iii) The boundary element least-squares term
$\ip{\V(\gamma\cdot)}{\V(\gamma\cdot)}_{H^{1/2}(\Gamma)}$ is of a standard type
and has to be approximated by using an appropriate preconditioner instead of the
$H^{1/2}(\Gamma)$-inner product. We give precise details in Section~\ref{sec_LS_impl} and
refer to \cite{GaticaHS_03_LSC,MaischakS_04_LSC,MaischakOS_12_LSF} for similar techniques.

(iv) It is also possible to define a least-squares coupling with equation~\eqref{BIEb} instead of~\eqref{BIEa}.
Due the fact that $\W(1,0)=0$, a (rank-one) stabilization term has to be added to ensure uniqueness. 
Here, the same term as for the coupling with Galerkin boundary elements can be used
(see Subsections~\ref{sec_W},~\ref{sec_V} below).
\end{remark}

\subsubsection{Coupling with Galerkin boundary elements (hypersingular operator)} \label{sec_W}
We proceed as in Subsection~\ref{sec_LS}, but couple equation~\eqref{BuL} with a variational form
of \eqref{BIEb} including transmission conditions instead of a least-squares form of \eqref{BIEa},
that is,
\begin{align*} 
  \dual{\W(\gamma\uu)}{\gamma_0\ww}_{\Gamma}
  &=
  \dual{\W(u_0,\phi_0)}{\gamma_0\ww}_{\Gamma}
  \quad\forall\ww\in U.
\end{align*}
Since $\W(1,0)=0$, a rank-one term has to be added for uniqueness. 
We therefore incorporate relation~\eqref{BIEa} combined with the transmission conditions,
and define the bilinear and linear forms
\begin{align*}
  c(\uu,\ww) &:= c_\BEW(\uu,\ww) := \dual{\W(\gamma\uu)}{\gamma_0\ww}_\Gamma
  + \dual{1}{\V(\gamma\uu)}_\Gamma\dual{1}{\V(\gamma\ww)}_\Gamma,\\
  L_U(\ww) &:= L_\BEW(\ww) := \dual{\W(u_0,\phi_0)}{\gamma_0\ww}_\Gamma
  + \dual{1}{\V(u_0,\phi_0)}_\Gamma\dual{1}{\V(\gamma\ww)}_\Gamma.
\end{align*}
Then, instead of \eqref{eq:lsc} we obtain the following coupling variant of \eqref{eq_dpgcoupling}:
\begin{align}\label{eq_W} 
  \uu\in U:\quad
  b(\uu,\Theta_\beta\ww) + c_\BEW(\uu,\ww)
  = L_V(\Theta_\beta\ww) + L_\BEW(\ww) \quad\forall \ww\in U.
\end{align}
Again, the bilinear form $b$ and functional $L_V$ are the ones from Section~\ref{sec_fem}.
Correspondingly, the discrete scheme is
\begin{align}\label{eq_W_discrete} 
  \uu_\hp\in U_\hp:\quad
  b(\uu_\hp,\Theta_\beta\ww) + c_\BEW(\uu_\hp,\ww)
  = L_V(\Theta_\beta\ww) + L_\BEW(\ww) \quad\forall \ww\in U_\hp.
\end{align}
It turns out that the bilinear form defined by the left-hand side of \eqref{eq_W} is $U$-elliptic
only for sufficiently large $\beta$. This is the reason for having introduced this parameter.
The following theorem is our second main result.

\begin{theorem}\label{thm_W}
  There exists a constant $\beta_0>0$ which depends only on $\Omega$
  such that for all $\beta\geq\beta_0$
  problems~\eqref{tp} and~\eqref{eq_W} are equivalent and uniquely solvable.
  More precisely, let $(u,u^c)$ be the solution to \eqref{tp} and define
  $\uu:=(u,\ssigma,\wat u,\wat\sigma)$ with $\ssigma:=\nabla u$ in $\Omega$,
  $\wat u:=u|_\cS$ (element-wise), and $\wat\sigma:=\ssigma\cdot\nn_\el$ on $\partial\el$
  for any $\el\in\TT$. Then, $\uu$ satisfies \eqref{eq_W}.
  On the other hand, if $\uu=(u,\ssigma,\wat u,\wat\sigma)$ solves \eqref{eq_W},
  then $(u,u^c)$ with $u^c:=\dlp(\wat u|_\Gamma-u_0) - \slp(\wat\sigma|_\Gamma - \phi_0)$
  solves \eqref{tp}.

  Furthermore, for a finite-dimensional subspace $U_\hp\subset U$, there is a unique 
  solution $\uu_\hp\in U_\hp$ of \eqref{eq_W_discrete}, and there holds the best approximation
  property
  \begin{align*}
    \norm{\uu-\uu_\hp}{U} \lesssim
    \beta \inf_{\ww\in U_\hp} \norm{\uu-\ww}{U}
  \end{align*}
  with hidden constant independent of $\TT$, $U_\hp$,
  and $\beta$ (under the restriction $\beta\ge\beta_0$).
\end{theorem}

A proof of this theorem is given in Section~\ref{sec_W_proof}.

\begin{remark}\label{rem_W}
(i) It can be shown that, for arbitrary $\beta>0$ but with upper bound for the diameter of $\Omega$,
    the bilinear form $b(\cdot,\Theta_\beta\,\cdot)+c_\BEW(\cdot,\cdot)$ from \eqref{eq_W}
    (with slightly different stabilization term) satisfies both inf-sup conditions at the continuous level.
    However, a proof of the discrete inf-sup condition is unknown. We therefore stick to
    an analysis based on ellipticity.

(ii) The corresponding versions of Remarks~\ref{rem2} (i), (ii) apply in this case as well.
\end{remark}

\subsubsection{Coupling with Galerkin boundary elements (single layer operator)} \label{sec_V}
We proceed again in the fashion of Subsection~\ref{sec_LS}. This time we couple equation~\eqref{BuL}
with a variational form of \eqref{BIEa} including the transmission conditions,
that is,
\begin{align*} 
  \dual{\gamma_\nn\ww}{\V(\gamma\uu)}_{\Gamma}
  &=
  \dual{\gamma_\nn\ww}{\V(u_0,\phi_0)}_{\Gamma}
  \quad\forall\ww\in U.
\end{align*}
We add the same rank-one term as before to ensure uniqueness, and define the bilinear and linear forms
\begin{align*}
  c(\uu,\ww) &:= c_\BEV(\uu,\ww) := \dual{\gamma_\nn\ww}{\V(\gamma\uu)}_\Gamma
  + \dual{1}{\V(\gamma\uu)}_\Gamma\dual{1}{\V(\gamma\ww)}_\Gamma,\\
  L_U(\ww) &:= L_\BEV(\ww) := \dual{\gamma_\nn\ww}{\V(u_0,\phi_0)}_\Gamma
  + \dual{1}{\V(u_0,\phi_0)}_\Gamma\dual{1}{\V(\gamma\ww)}_\Gamma.
\end{align*}
Then, instead of the Galerkin boundary element coupling with hypersingular integral operator \eqref{eq_W}
we obtain the following coupling variant of \eqref{eq_dpgcoupling}:
\begin{align}\label{eq_V} 
  \uu\in U:\quad
  b(\uu,\Theta_\beta\ww) + c_\BEV(\uu,\ww)
  = L_V(\Theta_\beta\ww) + L_\BEV(\ww) \quad\forall \ww\in U.
\end{align}
As before, the bilinear form $b$ and functional $L_V$ are the ones from Section~\ref{sec_fem}.
The discrete scheme is
\begin{align}\label{eq_V_discrete} 
  \uu_\hp\in U_\hp:\quad
  b(\uu_\hp,\Theta_\beta\ww) + c_\BEV(\uu_\hp,\ww)
  = L_V(\Theta_\beta\ww) + L_\BEV(\ww) \quad\forall \ww\in U_\hp.
\end{align}
For this coupling variant we require that the single layer operator $\slo$ is
$H^{-1/2}(\Gamma)$-elliptic. This is always true in three dimensions and, in two dimensions,
holds if the diameter of the domain is sufficiently small.
Main advantage of this coupling variant is that its analysis extends to operators of the
form $-\div A(x)\nabla u$ in \eqref{tpa} whose minimum eigenvalues are strictly bounded from below
by $1/4$, cf.~\cite{Sayas_09_VJN,Steinbach_11_NSO}. (To be precise, the analysis given in
Section~\ref{sec_V_proof} can be applied to such an operator. An improvement of the
technical estimate \eqref{viertel} there, leads to the bound for the minimum eigenvalue.)

\begin{theorem}\label{thm_V}
  In two dimensions ($d=2$) assume that $\diam(\Omega)<1$. Then, the
  statement of Theorem~\ref{thm_W} holds correspondingly for the coupling~\eqref{eq_V}
  and its discrete variant~\eqref{eq_V_discrete}.
\end{theorem}

\subsubsection{Coupling with Galerkin boundary elements (Calder\'on system)} \label{sec_C}

We now use both Calder\'on equations \eqref{BIE} to couple the exterior part of the transmission problem
to the interior ultra-weak formulation (this is~\eqref{BuL} in strong form).
The equations are
\begin{align*} 
  \dual{\W(\gamma\uu)}{\gamma_0\ww}_{\Gamma}
  &=
  \dual{\W(u_0,\phi_0)}{\gamma_0\ww}_{\Gamma}
  \quad\forall\ww\in U, \\
  \dual{\gamma_\nn\ww}{\V(\gamma\uu)}_{\Gamma}
  &=
  \dual{\gamma_\nn\ww}{\V(u_0,\phi_0)}_{\Gamma}
  \quad\forall\ww\in U.
\end{align*}
We add the same rank-one term as before to ensure uniqueness, and define the bilinear and linear forms
\begin{align*}
  c(\uu,\ww) &:= c_\BEC(\uu,\ww) := \dual{\W(\gamma\uu)}{\gamma_0\ww}_{\Gamma} 
  + \dual{\gamma_\nn\ww}{\V(\gamma\uu)}_\Gamma
  + \dual{1}{\V(\gamma\uu)}_\Gamma\dual{1}{\V(\gamma\ww)}_\Gamma,\\
  L_U(\ww) &:= L_\BEC(\ww) := \dual{\W(u_0,\phi_0)}{\gamma_0\ww}_\Gamma 
  + \dual{\gamma_\nn\ww}{\V(u_0,\phi_0)}_\Gamma
  + \dual{1}{\V(u_0,\phi_0)}_\Gamma\dual{1}{\V(\gamma\ww)}_\Gamma.
\end{align*}
Then, we obtain the following coupling variant of \eqref{eq_dpgcoupling}:
\begin{align}\label{eq_C} 
  \uu\in U:\quad
  b(\uu,\Theta_\beta\ww) + c_\BEC(\uu,\ww)
  = L_V(\Theta_\beta\ww) + L_\BEC(\ww) \quad\forall \ww\in U,
\end{align}
and discrete scheme
\begin{align}\label{eq_C_discrete} 
  \uu_\hp\in U_\hp:\quad
  b(\uu_\hp,\Theta_\beta\ww) + c_\BEC(\uu_\hp,\ww)
  = L_V(\Theta_\beta\ww) + L_\BEC(\ww) \quad\forall \ww\in U_\hp.
\end{align}
Again, the bilinear form $b$ and functional $L_V$ are the ones from Section~\ref{sec_fem}.
For this coupling variant, as for the case with weakly singular operator,
we need $H^{-1/2}(\Gamma)$-ellipticity of the single layer operator $\slo$. As previously noted,
in three dimensions this is always true whereas in two dimensions the domain has to be appropriately
scaled. Main advantage of this coupling variant is that its analysis does not make
use of special relations between Cauchy data of the interior and exterior problems
(specifically, Lemma~\ref{la_op} from Section~\ref{sec_LS_tech} below is not needed).
Its proof is therefore easier to extend to problems where the differential operator
in the interior is different from the one in the exterior. Also, we do not require a condition
on the minimum eigenvalue of operators of the type $-\div A(x)\nabla u$ for the interior problem
\eqref{tpa}, which is needed when extending the variant from Section~\ref{sec_V} to such operators.

\begin{theorem}\label{thm_C}
  In two dimensions ($d=2$) assume that $\diam(\Omega)<1$. Then, the statement of
  Theorem~\ref{thm_W} holds correspondingly for the variational formulation~\eqref{eq_C}
  and its discrete variant~\eqref{eq_C_discrete}.
\end{theorem}

\section{Analysis and proofs of the main results} \label{sec_LS_tech}

In this section we provide some technical details and prove Theorems~\ref{thm_equiv},~\ref{thm_W}
and~\ref{thm_V}.

For given $\varphi\in H^{1/2}(\Gamma)$ let us define its harmonic extension
\begin{align*}
   &\ext(\varphi) := \uu = (u,\ssigma,\wat u,\wat\sigma):
   \left\{\begin{array}{l}
   u\in H^1(\Omega):\; u=\varphi\ \text{on}\ \Gamma,\quad
   \vdual{\nabla u}{\nabla w}_\Omega = 0\quad \forall w\in H^1_0(\Omega),\\
   \ssigma = \nabla u,\quad
   \wat u = u\ \text{on}\ \cS,\quad
   \wat\sigma = \ssigma\cdot\nn_\el\ \text{on}\ \partial\el,\;\forall \el\in\TT.
   \end{array}\right.
\end{align*}
We can state the following properties of $\ext$.
\begin{lemma} \label{la_ext}
The operator $\ext:\; H^{1/2}(\Gamma)\to U$ is linear, bounded, and
a right-inverse of $\gamma_0$.
\end{lemma}
\begin{proof}
By construction we have $\gamma_0\ext\varphi=\varphi$ for any $\varphi\in H^{1/2}(\Gamma)$,
showing that $\ext$ is a right-inverse.
The continuity of $\ext$ follows from the well-posedness of the Dirichlet problem, i.e.,
$\|u\|_{H^1(\Omega)}\lesssim \|\varphi\|_{H^{1/2}(\Gamma)}$, the definition
of the other components of $\ext\varphi$ and the involved norms, cf.~\eqref{Hpm}
and Lemma~\ref{la_restriction}.
\end{proof}

\begin{lemma} \label{la_kerB}
There holds $\ker B = \ext H^{1/2}(\Gamma)$.
\end{lemma}

\begin{proof}
We note that $\ext H^{1/2}(\Gamma)\subset \ker B$ follows from integration by parts.
Therefore, it remains to show that $\ker B\subset \ext H^{1/2}(\Gamma)$.
Let $\uu\in U$ be arbitrary but fixed and define $\uu_0:=\uu-\ext\gamma_0\uu$.
Then $\gamma_0 \uu_0=0$ by Lemma~\ref{la_ext},
i.e., $\uu_0\in U_0$, and due to Lemma~\ref{lem_dg11} it holds
\begin{align}\label{pf_kerB_1}
  \begin{split}
  \norm{B\uu}{V'} &= \norm{B\uu_0}{V'}
   =
   \sup_{\vv\in V} \frac {b(\uu_0,\vv)}{\norm{\vv}{V}}
   \gtrsim
   \norm{\uu_0}{U}.
 \end{split}
\end{align}
Suppose that $\uu\in\ker B$. From~\eqref{pf_kerB_1} it follows that
$\uu = \ext\gamma_0\uu\in \ext H^{1/2}(\Gamma)$.
\end{proof}

\begin{lemma} \label{la_op}
  There holds $\V(\gamma\uu) = \gamma_0\uu$ and $\W(\gamma\uu) = \gamma_\nn\uu$
  for all $\uu\in\ker B$.
\end{lemma}

\begin{proof}
Let $\uu=(u,\ssigma,\wat u,\wat\sigma)\in\ker B$ be given. By Lemma~\ref{la_kerB} and the
definition of $\ext$ we have
\[
   \Delta u = 0\ \text{in}\ \Omega,\quad
   u=\wat u|_\Gamma\ \text{on}\ \Gamma,\quad\text{and}\quad
   \frac {\partial u}{\partial\nn_\Omega} = \wat\sigma|_\Gamma\ \text{on}\ \Gamma.
\]
Therefore, the representation formula (see \cite{McLean_00_SES}) yields
$u=\slp\wat\sigma-\dlp\wat u$ in $\Omega$, and application of the trace operator proves that
\(
   \wat u/2 = \slo\wat\sigma - \dlo\wat u
\)
on $\Gamma$, cf.~\eqref{slo_dlo}, that is,
\[
   \V(\gamma\uu) = \slo\wat\sigma + (1/2-\dlo)\wat u|_\Gamma = \wat u|_\Gamma = \gamma_0\uu.
\]
The relation involving $\W$ follows by a similar argument.
\end{proof}

The following norm equivalence follows by a standard compactness argument.
Similar results for classical coupling methods can be found in~\cite{AuradaFFKMP_13_CFB,dissTOFU}.

\begin{lemma}\label{la_normEquiv}
  In the case $d=2$ assume that $\diam(\Omega)<1$. There holds 
  \begin{align*}
    \norm{u}{H^{1/2}(\Gamma)}^2 + \norm{\phi}{H^{-1/2}(\Gamma)}^2 \simeq \dual{\hyp u}{u}_\Gamma +
    \dual\phi{\slo\phi}_\Gamma + |\dual{1}{\V(u,\phi)}_\Gamma|^2
  \end{align*}
  for all $(u,\phi)\in H^{1/2}(\Gamma)\times H^{-1/2}(\Gamma)$,
  and the involved constants only depend on $\Gamma$.
\end{lemma}

\subsection{Proof of Theorem~\ref{thm_equiv}} \label{sec_pf_equiv}
It is well known that \eqref{tp} is uniquely solvable, see, e.g., \cite{CostabelS_85_DBI}.
Let $(u,u^c)$ solve \eqref{tp} and let $\uu=(u,\ssigma,\wat u,\wat\sigma)$ be defined as
in the theorem. Then, by construction, $J_1(\uu;f,u_0,\phi_0)=0$, i.e., $\uu$ solves \eqref{LS}.
It remains to show that any solution $\uu$ to $\eqref{LS}$ is unique, and that $\uu_\hp$
satisfies the quasi-optimal error estimate. Both follow from Lemma~\ref{lem_dpgcoupling}.

To this end, we note that the right-hand side functionals $L_V(\Theta_1\,\cdot),\;L_\LS:\;U\to V'$
are bounded by the boundedness of $\Theta_1$ (see the proof of Lemma~\ref{lem_dpgcoupling}),
$L_V$ (since $f\in L_2(\Omega)$), $\V$ (by \eqref{op_bound}), and $\gamma$ (by Lemma~\ref{la_restriction}).
It remains to show the $U$-ellipticity of the bilinear form
$b(\cdot,\Theta_1\cdot)+c_\LS(\cdot,\cdot)$.

Let $\uu=(u,\ssigma,\wat u,\wat\sigma)\in U$ be given.
By Lemmas~\ref{lem_dg11},~\ref{la_ext}, and~\ref{la_kerB} we know that
\begin{align} \label{pf_equiv_1}
  \norm{B\uu}{V'} \gtrsim \norm{\uu-\ext\gamma_0\uu}{U}.
\end{align}
Furthermore, the continuity of the extension operator (see Lemma~\ref{la_ext}),
Lemmas~\ref{la_kerB},~\ref{la_op}, and the triangle inequality show that
\begin{align} \label{pf_equiv_2}
   \norm{\ext\gamma_0\uu}{U}
   &\lesssim
   \|\gamma_0\uu\|_{H^{1/2}(\Gamma)}
   =
   \|\V(\gamma\ext\gamma_0\uu)\|_{H^{1/2}(\Gamma)}
   \nonumber\\
   &\le
   \|\V(\gamma\uu)\|_{H^{1/2}(\Gamma)}
   +
   \|\V(\gamma(\uu-\ext\gamma_0\uu))\|_{H^{1/2}(\Gamma)}.
\end{align}
Now, using again the continuity of $\V$ and $\gamma$, estimate~\eqref{pf_equiv_1} yields
\begin{align} \label{pf_equiv_3}
   \|\V(\gamma(\uu-\ext\gamma_0\uu))\|_{H^{1/2}(\Gamma)}
   \lesssim
   \norm{\uu-\ext\gamma_0\uu}{U}
   \lesssim
   \|B\uu\|_{V'}.
\end{align}
Finally, combining \eqref{pf_equiv_1}--\eqref{pf_equiv_3} and the triangle inequality proves that
\begin{equation} \label{LS_ell}
   \norm{\uu}{U}
   \lesssim
   \|\V(\gamma\uu)\|_{H^{1/2}(\Gamma)} + \|B\uu\|_{V'}
   \quad\forall\uu\in U.
\end{equation}
Recalling the definition of $c_\LS(\cdot,\cdot)$ and the relation
$\|B\uu\|_{V'}^2=b(\uu,\Theta_1\uu)$, this proves the $U$-ellipticity of the bilinear form.\qed

\subsection{Proof of Theorem~\ref{thm_W}}\label{sec_W_proof}
We only have to slightly vary the proof of Theorem~\ref{thm_equiv} from Section~\ref{sec_pf_equiv}.

Integration by parts and the Calder\'on system \eqref{BIE} show that the solution $\uu$ of \eqref{tp}
also solves problem~\eqref{eq_W}.
It remains to prove that~\eqref{eq_W} is uniquely solvable and that discrete approximations
fulfill the best approximation property. Again we use Lemma~\ref{lem_dpgcoupling}.

In the current case, boundedness of $c(\cdot,\cdot)=c_\BEW(\cdot,\cdot)$
follows from the mapping properties of the integral operators $\V$ and $\W$
(see Section~\ref{sec_BIO}), and the boundedness of $\gamma$ and $\gamma_0$ by Lemma~\ref{la_restriction}.
It remains to prove $U$-ellipticity of the bilinear form
$b(\cdot,\Theta_\beta\,\cdot) + c(\cdot,\cdot)$.

We use the relations from Lemma~\ref{la_op} and the norm equivalence
$\norm{v}{H^1(\Omega)}^2 \simeq \norm{\nabla v}{L_2(\Omega)}^2 + |\dual{1}{v}_\Gamma|^2$ for $v\in H^1(\Omega)$.
Furthermore, let $\widetilde u$ denote the first component of $\EE\gamma_0\uu$.
Integration by parts and the harmonicity of $\widetilde u$ in $\Omega$ show that
$\dual{\nabla\widetilde u}{\nabla\widetilde u}_\Omega = \dual{\gamma_\nn\widetilde u}{\gamma_0\uu}_\Gamma$.
We obtain
\begin{align} \label{cW_ell}
  c(\EE\gamma_0\uu,\EE\gamma_0\uu)
  &=
  \dual{\hyp\wat u + (\tfrac12+\adlo)\gamma_\nn\EE\gamma_0\uu}{\wat u}_\Gamma
  + |\dual{1}{\V(\gamma\EE\gamma_0\uu)}_\Gamma|^2 \nonumber\\
  &= \dual{\gamma_\nn \EE\gamma_0\uu}{\wat u}_\Gamma + |\dual{1}{\wat u}_\Gamma|^2
  =
  \norm{\nabla\widetilde u}{L_2(\Omega)}^2
  + |\dual{1}{\wat u}_\Gamma|^2 \nonumber\\
  &\simeq
  \norm{\EE\gamma_0\uu}{H^1(\Omega)}^2
  \ge
  \norm{\gamma_0\uu}{H^{1/2}(\Gamma)}^2 \gtrsim \norm{\EE\gamma_0\uu}U^2.
\end{align}
By Lemmas~\ref{lem_dg11},~\ref{la_ext}, and~\ref{la_kerB} we bound
\begin{align*}
  \norm{\uu}{U}^2 \lesssim \norm{\uu-\ext\gamma_0\uu}{U}^2 + \norm{\ext\gamma_0\uu}{U}^2
                  \lesssim \norm{B\uu}{V'}^2 + \norm{\ext\gamma_0\uu}{U}^2.
\end{align*}
Combining the last two estimates, and using the boundedness of $c(\cdot,\cdot)$ and Young's inequality,
we conclude that, for $\delta>0$,
\begin{align*}
  \norm{\uu}U^2
  &\lesssim
  \norm{B\uu}{V'}^2 + \norm{\EE\gamma_0\uu}U^2
  \lesssim
  \norm{B\uu}{V'}^2 + c(\EE\gamma_0\uu,\EE\gamma_0\uu) \\
  &=
  \norm{B\uu}{V'}^2 + c(\EE\gamma_0\uu,\EE\gamma_0\uu-\uu) + c(\EE\gamma_0\uu-\uu,\uu) + c(\uu,\uu) \\
  &\lesssim
  \norm{B\uu}{V'}^2 + \norm{\EE\gamma_0\uu-\uu}U\norm{\uu}U + c(\uu,\uu) \\
  &\lesssim
  (1+\delta^{-1}) \norm{B\uu}{V'}^2 + \delta \norm{\uu}U^2 + c(\uu,\uu).
\end{align*}
Subtracting the last term for a sufficiently small $\delta>0$ this proves the existence of $\beta_0>0$
such that, for all $\beta\geq\beta_0$, there holds
$\norm{\uu}U^2 \lesssim \beta \norm{B\uu}{V'}^2 + c(\uu,\uu) = b(\uu,\Theta_\beta\uu) + c(\uu,\uu)$.
This finishes the proof of Theorem~\ref{thm_W}.
\qed

\subsection{Proof of Theorem~\ref{thm_V}}\label{sec_V_proof}

The proof follows the same lines as the proof of Theorem~\ref{thm_W} in Section~\ref{sec_W_proof}.
The only difference is that we have to prove ellipticity of $c(\cdot,\cdot) = c_\BEV(\cdot,\cdot)$ on
the kernel of $B$, cf.~\eqref{cW_ell}. 
To see this, we could again make use of the identity
$\slo\gamma_\nn\EE\gamma_0\uu = (\tfrac12+\dlo)\gamma_0\uu$ from Lemma~\ref{la_op}. 
However, applying techniques from~\cite{Sayas_09_VJN,Steinbach_11_NSO}, our proof extends
to more general operators in~\eqref{tpa}.

Let us abbreviate $\phi:=\gamma_\nn\EE\gamma_0\uu$ and let $\widetilde u\in H^1(\Omega)$
be the first component of $\EE\gamma_0\uu$. Then
$\dual{\phi}{\wat u}_\Gamma = \norm{\nabla \widetilde u}{L_2(\Omega)}^2$, and we calculate
\begin{align*}
  c(\EE\gamma_0\uu,\EE\gamma_0\uu)
  &=
  \dual{\gamma_\nn\EE\gamma_0\uu}{\slo\gamma_\nn\EE\gamma_0\uu + (\tfrac12-\dlo)\wat u}_\Gamma
  + |\dual{1}{\V(\gamma\EE\gamma_0\uu)}_\Gamma|^2 \\
  &= 
  \dual{\phi}{\slo\phi}_\Gamma 
  - \dual{(\tfrac12+\adlo)\phi}{\wat u}_\Gamma 
  + \dual{\phi}{\wat u}_\Gamma
  + |\dual{1}{\V(\gamma\EE\gamma_0\uu)}_\Gamma|^2 \\
  &= \dual\phi{\slo\phi}_\Gamma - \dual{(\tfrac12+\adlo)\phi}{\wat u}_\Gamma 
  + \norm{\nabla\widetilde u}{L_2(\Omega)}^2
  + |\dual{1}{\V(\gamma\EE\gamma_0\uu)}_\Gamma|^2.
\end{align*}
By using that $\partial_\nn\slp\psi = (\tfrac12+\adlo)\psi$ for $\psi\in H^{-1/2}(\Gamma)$,
it can be shown that
\begin{align} \label{viertel}
  - \dual{(\tfrac12+\adlo)\phi}{\wat u}_\Gamma \geq -\tfrac12 \dual{\phi}{\slo\phi}_\Gamma
  - \tfrac12\norm{\nabla\widetilde u}{L_2(\Omega)}^2.
\end{align}
For simplicity we refer to \cite[Proof of Theorem~9]{AuradaFFKMP_13_CFB} for a proof.
By the same reference, \cite[Lemma~10]{AuradaFFKMP_13_CFB}, there holds the norm equivalence
\[
   \dual{\phi}{\slo\phi}_\Gamma
   + \norm{\nabla\widetilde u}{L_2(\Omega)}^2
   + |\dual{1}{\V(\gamma\EE\gamma_0\uu)}_\Gamma|^2
   \simeq
   \norm{\widetilde u}{H^1(\Omega)}^2 + \norm{\phi}{H^{1/2}(\Gamma)}^2
\]
so that
\begin{align*}
  c(\EE\gamma_0\uu,\EE\gamma_0\uu)
  \geq
  \tfrac12 \dual{\phi}{\slo\phi}_\Gamma + \tfrac12\norm{\nabla\widetilde u}{L_2(\Omega)}^2
  + |\dual{1}{\V(\gamma\EE\gamma_0\uu)}_\Gamma|^2 \gtrsim \norm{\widetilde u}{H^1(\Omega)}^2.
\end{align*}
It follows that
$c(\EE\gamma_0\uu,\EE\gamma_0\uu)\gtrsim \norm{\EE\gamma_0\uu}U^2$, which finishes the proof.
\qed

\subsection{Proof of Theorem~\ref{thm_C}}\label{sec_C_proof}

As for the proof of Theorems~\ref{thm_W} and~\ref{thm_V} in Sections~\ref{sec_W_proof}
and~\ref{sec_V_proof}, respectively, we only have to show the $U$-ellipticity of the bilinear form
$b(\cdot,\Theta_\beta\cdot) + c_\BEC(\cdot,\cdot)$.
We now make use of Lemma~\ref{la_normEquiv} instead of Lemma~\ref{la_op}.
As in Section~\ref{sec_W_proof}, we have
\begin{align*}
  \norm{\uu}U^2 \leq \norm{\uu-\EE\gamma_0\uu}U^2 + \norm{\EE\gamma_0\uu}U^2 \lesssim \norm{B\uu}{V'}^2 +
  \norm{\gamma_0\uu}{H^{1/2}(\Gamma)}^2.
\end{align*}
Applying Lemma~\ref{la_normEquiv}, we deduce that
\begin{align*}
\lefteqn{
  \norm{\gamma_0\uu}{H^{1/2}(\Gamma)}^2
  \le
  \norm{\gamma_0\uu}{H^{1/2}(\Gamma)}^2
  + \norm{\gamma_\nn\uu}{H^{-1/2}(\Gamma)}^2
} \\
  &\simeq
  \dual{\hyp \gamma_0\uu}{\gamma_0\uu}_\Gamma
  + \dual{\gamma_\nn\uu}{\slo\gamma_\nn\uu}_\Gamma
  + |\dual{1}{\V\gamma\uu}_\Gamma|^2 \\
  &=
  \dual{\hyp\gamma_0\uu + (\tfrac12+\adlo)\gamma_\nn\uu}{\gamma_0\uu}_\Gamma 
  + \dual{\gamma_\nn\uu}{(\tfrac12-\dlo)\gamma_0\uu+\slo\gamma_\nn\uu}_\Gamma
  + |\dual{1}{\V\gamma\uu}_\Gamma|^2
  - \dual{\gamma_\nn\uu}{\gamma_0\uu}_\Gamma \\
  &=
  c_\BEC(\uu,\uu)
  - \dual{\gamma_\nn(\uu-\EE\gamma_0\uu)}{\gamma_0\uu}_\Gamma
  - \dual{\gamma_\nn\EE\gamma_0\uu}{\gamma_0\uu}_\Gamma.
\end{align*}
Since $\EE\gamma_0$ is a harmonic function there holds $\dual{\gamma_\nn\EE\gamma_0\uu}{\gamma_0\uu}_\Gamma 
= \dual{\nabla\widetilde u}{\nabla\widetilde u}_\Omega \geq 0$
with $\widetilde u$ being the first component of $\EE\gamma_0\uu \in U$. We end up with
\begin{align*}
  \norm{\gamma_0\uu}{H^{1/2}(\Gamma)}^2 \lesssim 
  c_\BEC(\uu,\uu) -\dual{\gamma_\nn(\uu-\EE\gamma_0\uu)}{\gamma_0\uu}_\Gamma 
  \lesssim c_\BEC(\uu,\uu) + \norm{\uu-\EE\gamma_0\uu}U\norm{\uu}U.
\end{align*}
To finish the proof, we argue as in the proof of Theorem~\ref{thm_W}, i.e.,
applying Young's inequality and the continuity $\norm{\uu-\EE\gamma_0\uu}U\lesssim \norm{B\uu}{V'}$.
\qed

\section{Numerical experiments}\label{sec_num}
In this section we present numerical experiments in two dimensions for the three coupling schemes
from Sections~\ref{sec_LS} (with least-squares BEM),~\ref{sec_W} (with hypersingular operator),
and~\ref{sec_V} (with weakly singular operator). We also discuss their implementation.
Numerical results for the coupling with both Calder\'on equations from Section~\ref{sec_C}
are very similar to the ones of the other variants and are not reported.

Throughout, we consider regular and quasi-uniform triangulations with compact triangles.
For uniform refinements we have a sequence of uniform meshes
$\TT_0,\dots,\TT_L =: \TT$ of $\Omega$. Here, $\TT_\ell$ is the uniform refinement of $\TT_{\ell-1}$,
i.e., every triangle from $\TT_{\ell-1}$ is split into four triangles.
Moreover, $\cS_\ell$ denotes the skeleton induced by $\TT_\ell$ and, as before, $\cS$ is the skeleton of $\TT$.
The mesh-sizes $h_0\geq\dots\geq h_L=:h$ denote the largest diameters of the elements in
$\TT_\ell$ ($l=0,\ldots,L$).  Note that there holds $h_\ell = 2^{-\ell}h_0$.

Let $P^p(\TT_\ell)$ be the space of $\TT_\ell$-elementwise polynomials of degree $\leq p$,
and $S^1(\cS_\ell)$ the space of edge-wise affine and globally continuous functions on $\cS_\ell$.
Furthermore, $P^0(\cS_\ell)$ denotes the space of edge-wise constant functions on $\cS_\ell$.
Then we choose the lowest-order discrete trial spaces
\[
  U_\ell := P^0(\TT_\ell) \times \left[P^0(\TT_\ell)\right]^d \times 
  S^1(\cS_\ell) \times P^0(\cS_\ell)
\]
and our ansatz space for all variants is $U_\hp:=U_L$.

The implementation of the trial-to-test operator $\Theta_\beta:U\to V$ requires an approximation.
A standard way is to consider an enriched subspace $\widetilde V\subset V$ that uses the same mesh
$\TT$ as $U_\hp$ but is enriched by increasing polynomial degrees.
We select $\widetilde V:=P^2(\TT) \times \left[ P^2(\TT) \right]^2$ and consider, instead
of $\Theta_\beta$, the approximated trial-to-test operator
\begin{align*}
  \widetilde\Theta:\; U_\hp\to\widetilde V:\quad
  \ip{\widetilde\Theta\ww}{\vv}_V = b(\ww,\vv) \quad\forall\vv\in \widetilde V.
\end{align*}
Note that we have selected the weight $\beta=1$ and, in fact, all our numerical results are based on this
choice. For an analysis of the approximation of the trial-to-test operator we refer to
\cite{GopalakrishnanQ_14_APD}. There, it is shown that the increase of polynomial degrees by two
in two dimensions (to generate the enriched space $\widetilde V$) maintains validity of the inf-sup condition.

The boundary integral operators appearing in the Galerkin bilinear forms of our coupling schemes
are implemented in a standard way.
However, implementation of the least-squares bilinear forms involving boundary integral operators
is more complicated. This is discussed in the next section.
Throughout, our examples are conducted in \textsc{MATLAB} and for the discretization of
boundary integral operators we use the library \textsc{HILBERT}~\cite{AuradaEFFFGKMP_HMI}.

\subsection{Implementation of the coupling with least-squares BEM}\label{sec_LS_impl}

An implementation of our least-squares coupling scheme faces two problems: both the
operator $\V(\gamma(\cdot))\in H^{1/2}(\Gamma)$ and the $H^{1/2}(\Gamma)$-inner product
must be approximated.

First we discuss an approximation of $\V(\gamma(\cdot))$.
Let $\Pi_\ell: L_2(\Gamma) \to S^1(\cS_\ell|_\Gamma)$ denote the $L_2$-orthogonal projection.
Here, $S^1(\cS_\ell|_\Gamma)$ is the restriction of $S^1(\cS_\ell)$ to $\Gamma$. We set
$\Pi_h:=\Pi_L$ and replace the bilinear form
\[
   b(\uu,\Theta_1\ww) + c_\LS(\uu,\ww)
   = \ip{B\uu}{B\ww}_{V'} + \ip{\V(\gamma(\uu))}{\V(\gamma(\ww))}_{H^{1/2}(\Gamma)}
\]
(cf.~\eqref{ultra_bil_L},~\eqref{eq:lsc}) by the discretized bilinear form
\begin{align*}
  a_h(\uu,\ww)
  :=
  \ip{B\uu}{B\ww}_{V'} + \ip{\Pi_h\V(\gamma(\uu))}{\Pi_h\V(\gamma(\ww))}_{H^{1/2}(\Gamma)}
  \quad\forall\uu,\ww\in U_{hp}.
\end{align*}
In the same manner, we replace the right-hand side functional
\[
   L_V(\Theta_1\ww)+L_\LS(\ww)
   =
   \ip{L_V}{B\ww}_{V'} + \ip{\V(u_0,\phi_0)}{\V(\gamma(\ww))}_{H^{1/2}(\Gamma)}
\]
by
\begin{align*}
  L_h(\ww) :=
  \ip{L_V}{B\ww}_{V'} + \ip{\Pi_h\V(u_0,\phi_0)}{\Pi_h\V(\gamma(\ww))}_{H^{1/2}(\Gamma)}
  \quad\forall\ww\in U_{hp}.
\end{align*}
Then, our discrete DPG scheme coupled with least-squares boundary elements is
\begin{align}\label{eq:LS:discrete}
  \uu_\hp\in U_\hp:\quad a_h(\uu_\hp, \ww) = L_h(\ww) \quad\forall\ww \in U_\hp.
\end{align}
As previously mentioned, we also approximate the trial-to-test operator but this is not analyzed here.
Quasi-optimality of the discrete scheme is maintained, as the following theorem shows.

\begin{theorem}\label{thm:LS:discrete}
  Let $P_h: H^{1/2}(\Gamma)\to S^1(\cS|_\Gamma)$ denote any bounded projection, i.e.,
  there is a $\TT$-independent constant $C>0$ such that
  \[
    \norm{P_hu}{H^{1/2}(\Gamma)} \leq C \norm{u}{H^{1/2}(\Gamma)}
    \quad\forall u\in H^{1/2}(\Gamma) \quad\text{and}\quad 
    P_h u_h = u_h \quad\forall u_h \in S^1(\cS|_\Gamma).
  \]
  Then, $|a_h(\uu,\ww)| \lesssim \norm{\uu}U\norm{\vv}U$ for all $\uu,\ww\in U$ and
  \begin{align}\label{eq:LS:equiv}
    \norm{\uu}U \simeq \norm{B\uu}{V'} + \norm{P_h\V(\gamma\uu)}{H^{1/2}(\Gamma)} 
    \quad\forall \uu\in U_\hp.
  \end{align}
  The involved constants do not depend on $\TT$.

  In particular, one can choose the $L_2$-orthogonal projection $P_h = \Pi_h$.
\end{theorem}

\begin{proof}
  First we note that the boundedness of $P_h$ implies boundedness of $|a_h(\cdot,\cdot)|$.
  Furthermore, the boundedness of $P_h$ and stability of the continuous problem imply
  \begin{align*}
     \norm{B\uu}{V'} + \norm{P_h\V(\gamma\uu)}{H^{1/2}(\Gamma)} 
     \lesssim \norm{B\uu}{V'} + \norm{\V(\gamma\uu)}{H^{1/2}(\Gamma)} 
     \simeq \norm{\uu}U \quad\forall\uu\in U_\hp.
  \end{align*}
  The remainder of the proof is a slight modification of the proof of Theorem~\ref{thm_equiv}.
  By the projection property of $P_h$ we have
  \begin{align*}
    \norm{\uu}U \lesssim \norm{B\uu}{V'} + \norm{\gamma_0 \uu}{H^{1/2}(\Gamma)} 
    = \norm{B\uu}{V'} + \norm{P_h \gamma_0 \uu}{H^{1/2}(\Gamma)}
    \quad\forall\uu\in U_\hp.
  \end{align*}
  Relation $\gamma_0 \uu_{hp} = \V(\gamma\ext\gamma_0\uu_{hp})$ by Lemma~\ref{la_op} leads to
  \begin{align*}
    \norm{P_h \gamma_0 \uu}{H^{1/2}(\Gamma)} &= \norm{P_h \V(\gamma\ext\gamma_0 \uu)}{H^{1/2}(\Gamma)}
    \le \norm{P_h \V(\gamma\uu)}{H^{1/2}(\Gamma)} 
    + \norm{P_h \V(\gamma\ext\gamma_0\uu-\gamma\uu)}{H^{1/2}(\Gamma)}  \\
    &\le \norm{P_h \V(\gamma\uu)}{H^{1/2}(\Gamma)} + \norm{\uu-\ext\gamma_0 \uu}U
    \quad\forall\uu\in U_\hp.
  \end{align*}
  Hence the bound $\norm{\uu-\ext\gamma_0 \uu}U \lesssim \norm{B\uu}{V'}$
  finishes the proof of~\eqref{eq:LS:equiv}.

  Finally, boundedness of the $L_2$-projection in $H^{1/2}(\Gamma)$ is discussed, for instance,
  in~\cite{kpp} and the references given there.
\end{proof}

\begin{cor}\label{cor:LS:discrete}
  Let the assumptions of Theorem~\ref{thm:LS:discrete} be satisfied.
  The discrete problem~\eqref{eq:LS:discrete} admits a unique solution $\uu_{hp}\in U_{hp}$
  and, with $\uu\in U$ being the solution of~\eqref{eq:lsc}, there holds
  \begin{align*}
    \norm{\uu-\uu_{hp}}U \lesssim \inf_{\ww\in U_{hp}} \norm{\uu-\ww}U.
  \end{align*}
\end{cor}
\begin{proof}
  By Theorem~\ref{thm:LS:discrete} we can apply the Lax-Milgram lemma.
  It proves that~\eqref{eq:LS:discrete} has a unique solution.
  The second Strang lemma (see, e.g.,~\cite[Lemma~2.25]{ErnGuermond2014}) shows that
  \begin{align*}
    \norm{\uu-\uu_{hp}}U \lesssim \inf_{\ww\in U_{hp}} \norm{\uu-\ww}U + \sup_{\ww\in U_{hp}}
    \frac{|a_h(\uu,\ww)-L_h(\ww)|}{\norm{\ww}U}.
  \end{align*}
  The definitions of $a_h(\cdot,\cdot)$ and $L_h$ infer that
  \begin{align*}
    a_h(\uu,\ww)-L_h(\ww) = \ip{B\uu-L}{B\ww}_{V'} +
    \ip{P_h\V(\gamma\uu-(u_0,\phi_0))}{P_h\V(\gamma\ww)}_{H^{1/2}(\Gamma)}.
  \end{align*}
  Since $\uu$ is the exact (unique) solution, we know by Theorem~\ref{thm_equiv}
  that it satisfies equations~\eqref{BuL} and~\eqref{op}, that is,
  $B\uu=L$ and $\V(\gamma\uu)=\V(u_0,\phi_0)$.
\end{proof}

It remains to discuss some implementational aspects.
Let $\{\uu_j;\; j=1,\dots,\dim(U_\hp)\}$ and $\{\eta_j;\; j=1,\ldots,\dim(S^1(\cS_\ell|_\Gamma))\}$
denote, respectively, bases of $U_\hp$ and $S^1(\cS_\ell|_\Gamma)$.
Define the matrices
\begin{align*}
  \Omat_{jk} &:= \dual{\eta_j}{\V(\gamma\uu_k)}_\Gamma, \quad
  \Mmat_{jk} := \dual{\eta_j}{\eta_k}_\Gamma.
\end{align*}
A matrix representation (of the coefficients in the corresponding bases) of
$\Pi_h \V(\gamma(\cdot))$ is then given by $\Mmat^{-1} \Omat$.
In order to replace the $H^{1/2}(\Gamma)$-inner product for functions
$u_h\in S^1(\cS_\ell|_\Gamma)$ let $\Pmat$ denote a symmetric and positive definite matrix such that
\begin{align*} 
  \xx^T \Pmat \xx
  \simeq
  \norm{u_h}{H^{1/2}(\Gamma)}^2 \quad\forall u_h\in S^1(\cS_\ell|_\Gamma)
  \text{ with coefficient vector } \xx.
\end{align*}
Let $\xx,\yy$ denote the coefficient vectors corresponding to $u_h,w_h\in S^1(\cS_\ell|_\Gamma)$,
respectively. Then, instead of the evaluation of $\ip{u_h}{v_h}_{H^{1/2}(\Gamma)}$, we use $\yy^T\Pmat\xx$.
In particular, we replace the evaluation of
$\ip{\Pi_h\V(\gamma(\uu_{hp}))}{\Pi_h\gamma(\ww_{hp})}_{H^{1/2}(\Gamma)}$ with its matrix form
\[
  \yy^T(\Omat^{T} \Mmat^{-1} \Pmat  \Mmat^{-1}\Omat\xx).
\]
Different choices of $\Pmat$ are possible.
For instance, it is well known that the following stabilization of the hypersingular integral operator satisfies
\begin{align*}
  \dual{\hyp u_h}{u_h}_\Gamma + |\dual{u_h}1_\Gamma|^2 \simeq \norm{u_h}{H^{1/2}(\Gamma)}^2
  \quad\forall u_h\in S^1(\cS_\ell|_\Gamma).
\end{align*}
We define the corresponding inner product as
$\ip{u_h}{v_h}_\hyp := \dual{\hyp u_h}{v_h}_\Gamma + \dual{u_h}1_\Gamma\dual{v_h}1_\Gamma$.

Another possibility is to use multilevel norms.
By \cite[Theorem~1]{os98} and references given there, we have the equivalence
\[
  \norm{u}{H^{1/2}(\Gamma)}^2 \simeq \sum_{\ell=0}^\infty 2^\ell \norm{(\Pi_\ell-\Pi_{\ell-1})u}{L_2(\Gamma)}^2
  \quad\forall u\in H^{1/2}(\Gamma).
\]
Here, $\Pi_{-1}:=0$ and $\Pi_\ell$ is the $L_2$-orthogonal projection onto
$S^1(\cS_\ell|_\Gamma)$ for $\ell\geq 0$ and, as previously defined, $\TT_0,\TT_1,\dots$ is a sequence of
uniformly refined triangulations. Obviously, if $u$ is discrete, the above sum is finite. In particular, we have
\[
  \norm{u_h}{H^{1/2}(\Gamma)}^2 \simeq \sum_{\ell=0}^L 2^\ell \norm{(\Pi_\ell-\Pi_{\ell-1})u_h}{L_2(\Gamma)}^2
  \quad\forall u_h\in S^1(\cS_L|_\Gamma)
\]
with corresponding multilevel inner product
\begin{align*}
  \ip{u_h}{v_h}_\mathrm{ML} := \sum_{\ell=0}^L 2^\ell \dual{(\Pi_\ell-\Pi_{\ell-1})u_h}{v_h}_\Gamma \quad\forall
  u_h,v_h \in S^1(\cS_L|_\Gamma).
\end{align*}
Then, $\ip{u_h}{u_h}_\mathrm{ML} \simeq \norm{u_h}{H^{1/2}(\Gamma)}^2$ for all $u_h\in S^1(\cS_L|_\Gamma)$.

\begin{remark}
  (i) Theorem~\ref{thm:LS:discrete} and Corollary~\ref{cor:LS:discrete} remain valid if we replace the
  $H^{1/2}(\Gamma)$-inner product 
  in the definition of $a_h(\cdot,\cdot)$ and $L_h(\cdot)$ by $\ip\cdot\cdot_{\hyp}$ or $\ip\cdot\cdot_\mathrm{ML}$.

  (ii) The Galerkin matrix of $a_h(\cdot,\cdot)$ is symmetric and positive definite. Thus, standard
  iterative schemes such as the CG algorithm can be used.

  (iii) Since the bilinear forms $c_\BEW, c_\BEV, c_\BEC$ are not symmetric,
  it follows that the Galerkin matrices of $b(\cdot,\Theta_\beta(\cdot)) + c(\cdot,\cdot)$ with
  $c\in\{c_\BEW, c_\BEV, c_\BEC\}$ are non-symmetric.
  However, their symmetric parts are positive definite and the GMRES method can be used as solver.
\end{remark}

\subsection{A simple error estimation} \label{sec_est}

In our implementation, we approximate the data $u_0,\phi_0$ by their $L_2$-projections
$u_{0h}\in S^1(\cS|_\Gamma)$, $\phi_{0h}\in P^0(\cS|_\Gamma)$. Then we can use the discretized
boundary integral operators for the evaluation of the right-hand sides.
Clearly, this induces additional consistency errors (usually called oscillations).
They are of standard type and are therefore neglected in the following considerations.

The next result provides us with a reliable bound for the distance of any discrete function to the exact solution.

\begin{theorem}
  Let the assumptions of Theorem~\ref{thm:LS:discrete} be satisfied, let $u_{0h}\in S^1(\cS|_\Gamma)$,
  $\phi_{0h}\in P^0(\cS|_\Gamma)$ be approximations of the given data $u_0,\phi_0$,
  and let $\uu$ be the exact solution of problem~\eqref{tp} as defined in Theorem~\ref{thm_equiv}
  with data $u_0,\phi_0$ being replaced by $u_{0h},\phi_{0h}$. Then,
  \begin{align*}
    \norm{\uu-\ww_{hp}}U \lesssim \norm{B\ww_{hp}-L}{V'} &+ 
    \norm{P_h \V((u_{0h},\phi_{0h})-\gamma\ww_{hp})}{H^{1/2}(\Gamma)} \\
    &+ \norm{h^{1/2}\nabla_\Gamma \V((u_{0h},\phi_{0h})-\gamma\ww_{hp})}{L_2(\Gamma)}
    \qquad\forall\ww_{hp} \in U_{hp}
  \end{align*}
  with $\TT$-independent constants.
\end{theorem}

\begin{proof}
  The boundedness and $U$-ellipticity of $b(\cdot,\Theta_1\,\cdot)+c_\LS(\cdot,\cdot)$
  (cf.~\eqref{LS_ell}) show
  \begin{align*}
    \norm{\uu-\ww_{hp}}U &\simeq \norm{B(\ww_{hp}-\uu)}{V'} + 
    \norm{\V(\gamma\uu-\gamma\ww_{hp})}{H^{1/2}(\Gamma)} \\
    &\leq \norm{B(\ww_{hp}-\uu)}{V'} + \norm{P_h\V(\gamma\uu-\gamma\ww_{hp})}{H^{1/2}(\Gamma)} \\
    &\qquad\qquad+ \norm{(1-P_h)\V(\gamma\uu-\gamma\ww_{hp})}{H^{1/2}(\Gamma)}.
  \end{align*}
  Using the fact that $\uu$ is the exact solution, we have $B\uu=L$ and $\V(\gamma\uu) = \V(u_{0h},\phi_{0h})$.
  Then, applying~\cite[Lemma~24]{arcme}, that is,
  \begin{align*}
    \norm{(1-P_h)\V( (u_{0h},\phi_{0h})-\gamma\ww_{hp})}{H^{1/2}(\Gamma)} \lesssim 
    \norm{h^{1/2}\nabla_\Gamma \V( (u_{0h},\phi_{0h})-\gamma\ww_{hp})}{L_2(\Gamma)},
  \end{align*}
  we finish the proof.
\end{proof}

 From the last theorem we infer that, up to data oscillation terms, we have a computable error bound
\begin{align} \label{est}
  \est(\uu_{hp}):=\est(\uu_{hp};f,u_{0h},\phi_{0h}) :=  \norm{B\uu_{hp}-L}{V'} &+ 
    \norm{\Pi_h \V((u_{0h},\phi_{0h})-\gamma\uu_{hp})}{H^{1/2}(\Gamma)} \nonumber\\
    &+\norm{h^{1/2} \nabla_\Gamma \V((u_{0h},\phi_{0h})-\gamma\uu_{hp})}{L_2(\Gamma)}.
\end{align}
In particular, the first term is the usual energy error in the DPG method and the second term
can be evaluated as described in Section~\ref{sec_LS_impl} above.
The last term has the same form as a weighted residual error estimator in BEM, see~\cite{arcme} for an overview.
Note that $\est(\cdot)$ can be used to estimate both the error of the least-squares coupling and of the
coupling methods with Galerkin BEM.

\begin{remark}
  For the coupling with Galerkin boundary elements (at least for the case with hypersingular operator
  from Subsection~\ref{sec_W}), reliable and localizable residual error estimators can be
  developed, following the lines of~\cite{AuradaFFKMP_13_CFB} and references given there.
  However, for brevity and better comparability of the least-squares and Galerkin boundary element couplings,
  we use $\est(\cdot)$ from \eqref{est} to plot an upper bound for the global error
  (comprising all solution components) for all our coupling methods.
\end{remark}

\subsection{Example with smooth solution}

\begin{figure}
  \begin{center}
    \psfrag{x}[c][c]{\tiny $x$}
    \psfrag{y}[c][c]{\tiny $y$}

    \psfrag{number of elements}[c][c]{\tiny number of elements $N$}
    \psfrag{error and estimator}[c][c]{\tiny error and estimator}
    \psfrag{errU}{\tiny $\norm{u-u_{hp}}{L_2(\Omega)}$}
    \psfrag{errSigma}{\tiny $\norm{\ssigma-\ssigma_{hp}}{L_2(\Omega)}$}
    \psfrag{errEst}{\tiny $\est(\uu_{hp})$}
    \psfrag{Nm12}{\tiny $N^{-1/2}$}

    \includegraphics[height=0.4\textwidth,width=0.32\textwidth]{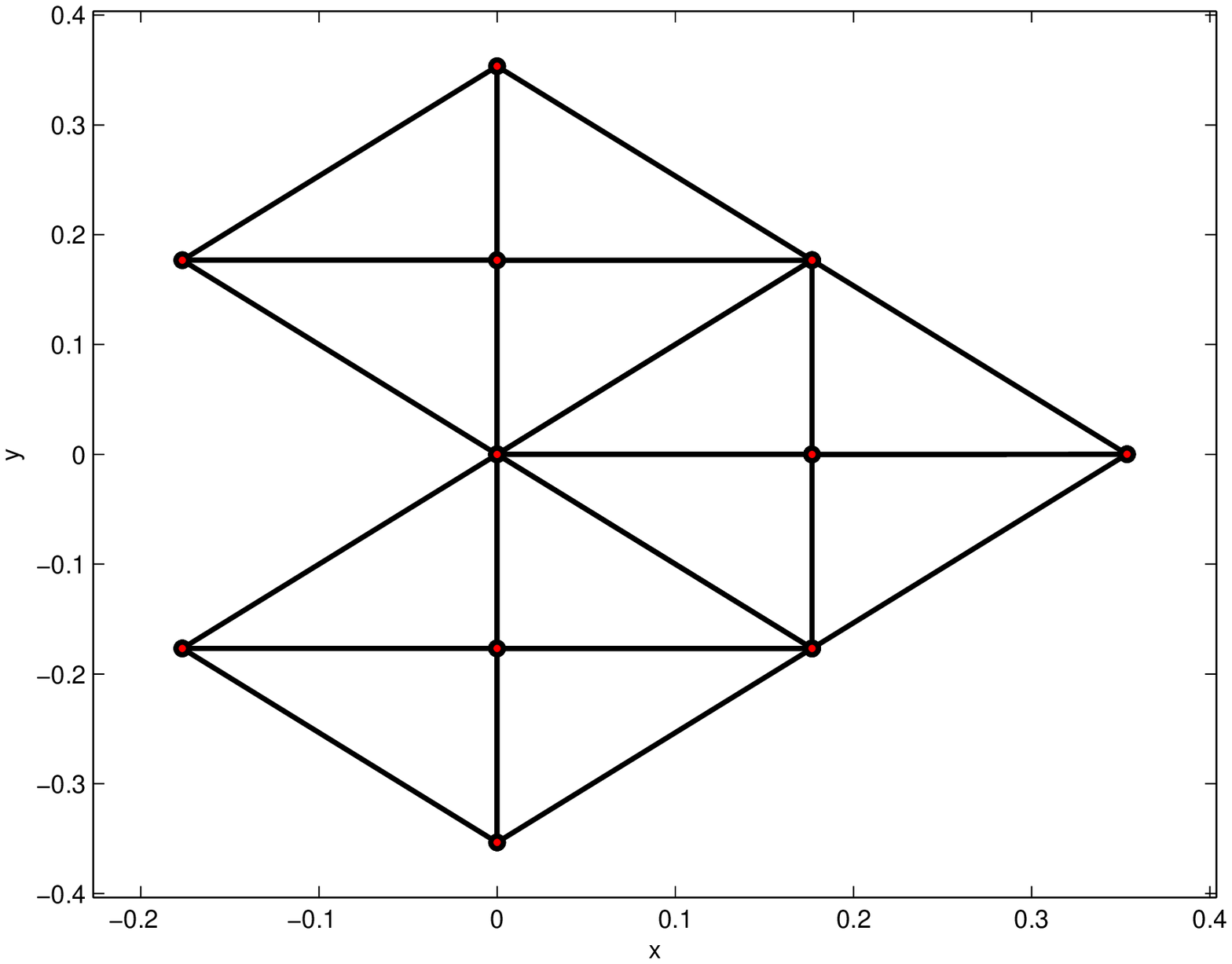}
    \qquad
    \includegraphics[width=0.45\textwidth]{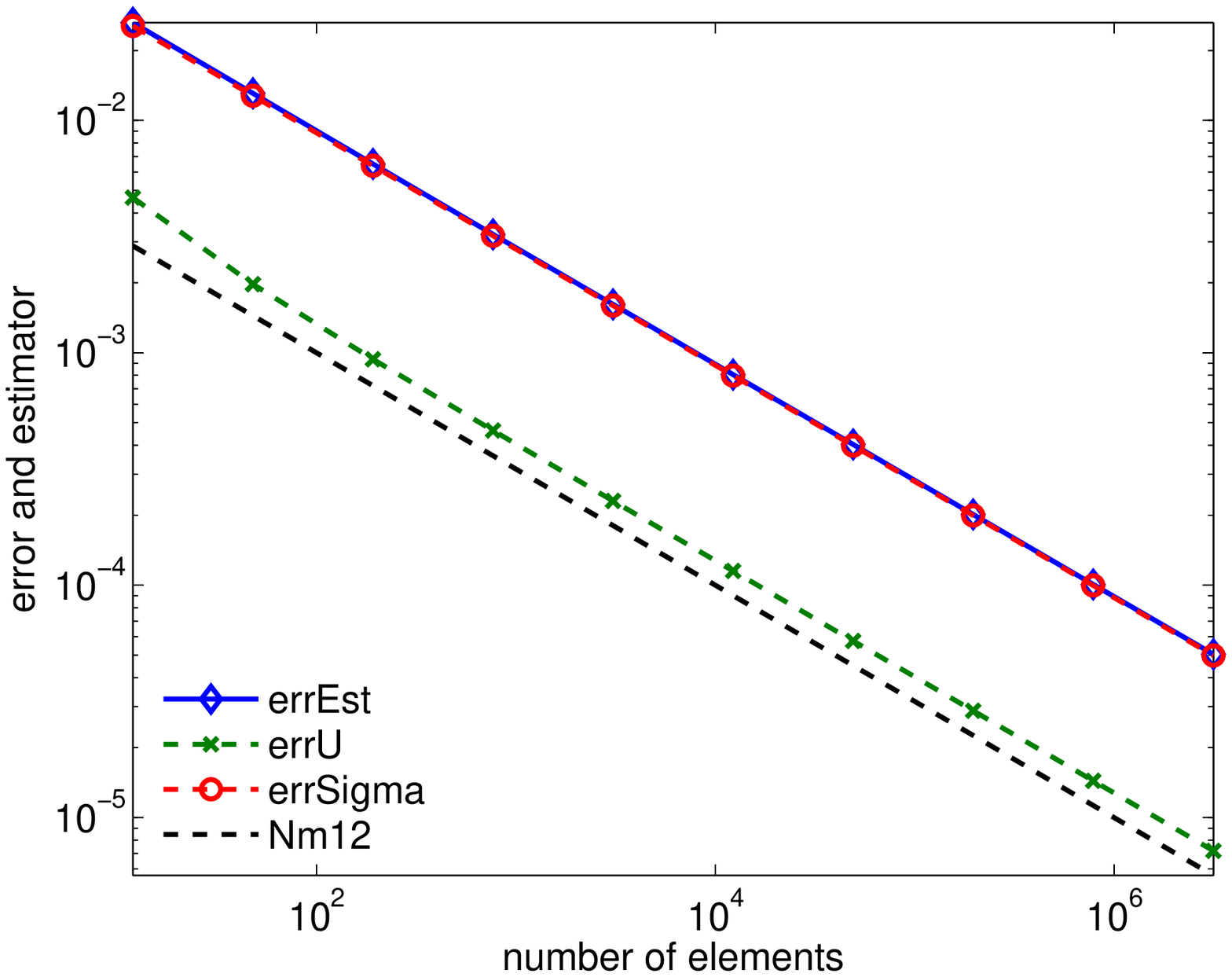}

  \end{center}
  \caption{Domain with initial triangulation $\TT_0$ (left); $L_2$ errors of $u$ and $\ssigma$,
           and error estimator est$(\cdot)$ for the least-squares coupling (right), with smooth solution $u$.}
  \label{fig:Lshape}
\end{figure}

\begin{figure}
  \begin{center}

    \psfrag{number of elements}[c][c]{\tiny number of elements $N$}
    \psfrag{error and estimator}[c][c]{\tiny error and estimator}
    \psfrag{errU}{\tiny $\norm{u-u_{hp}}{L_2(\Omega)}$}
    \psfrag{errSigma}{\tiny $\norm{\ssigma-\ssigma_{hp}}{L_2(\Omega)}$}
    \psfrag{errEst}{\tiny $\est(\uu_{hp})$}
    \psfrag{Nm12}{\tiny $N^{-1/2}$}

    \includegraphics[width=0.45\textwidth]{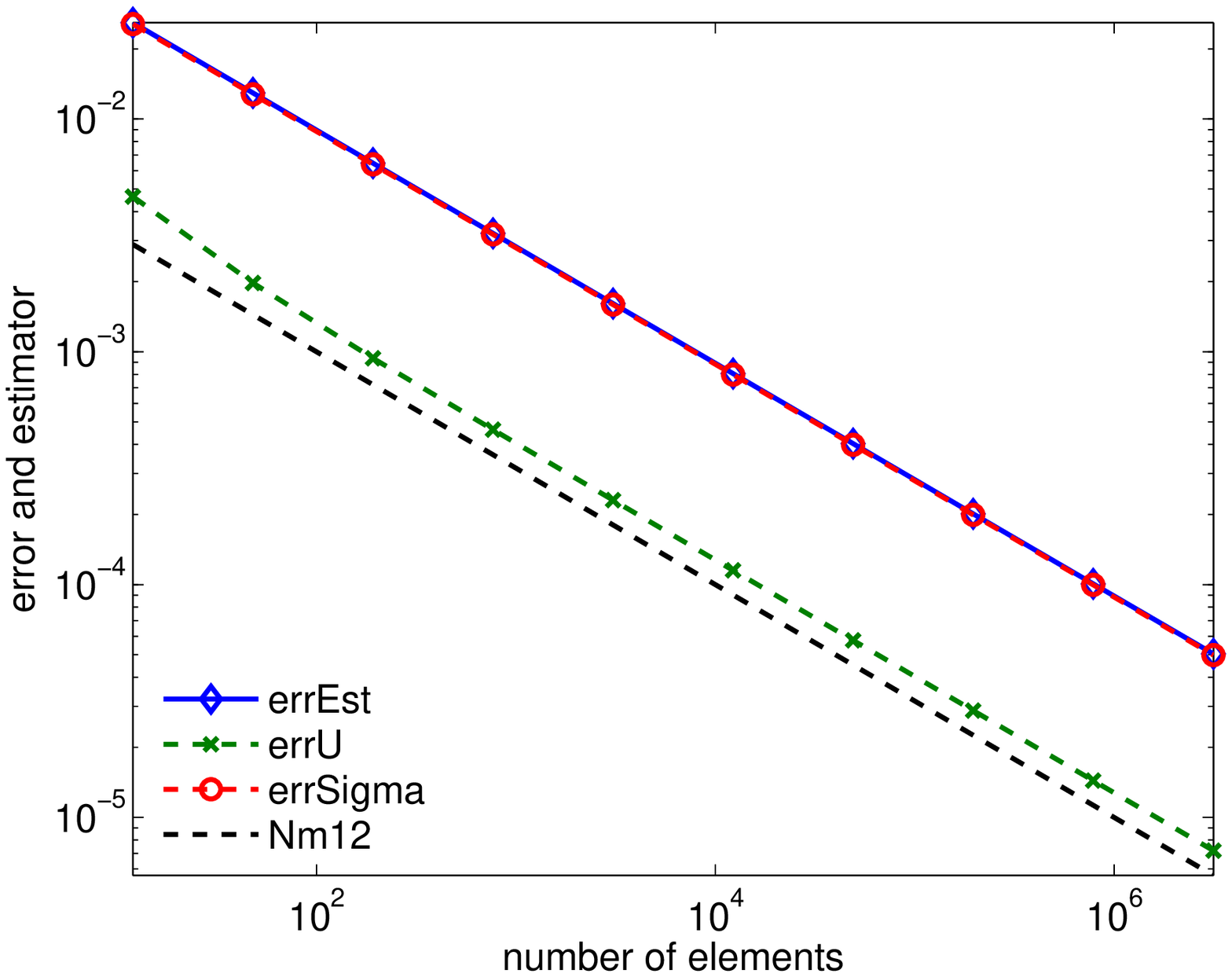}
    \quad
    \includegraphics[width=0.45\textwidth]{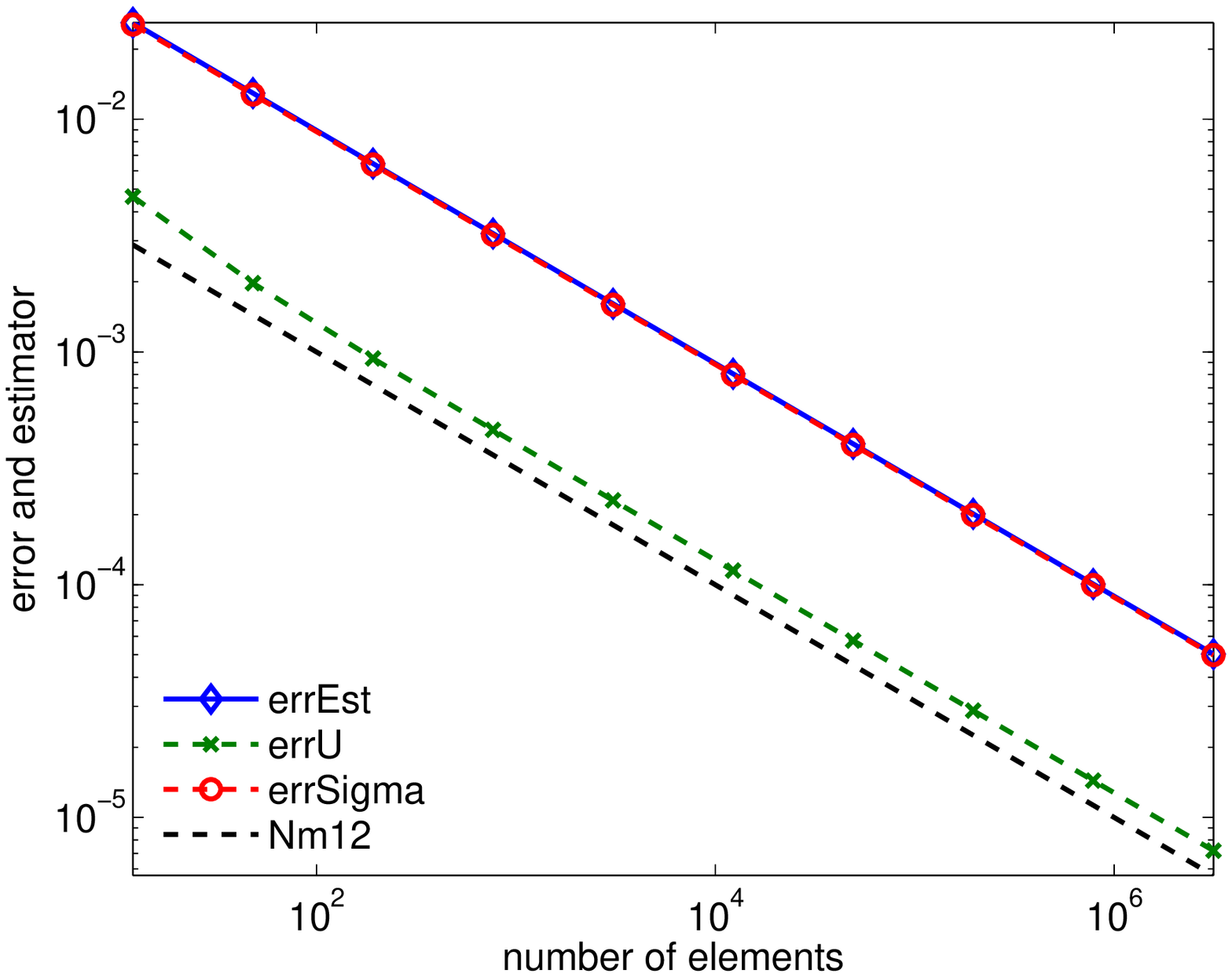}
  \end{center}
  \caption{$L_2$ errors of $u$ and $\ssigma$, and estimator est$(\cdot)$ for coupling with Galerkin BEM and smooth
           solution $u$: $c(\cdot,\cdot)=c_\BEW(\cdot,\cdot)$ (left) and $c(\cdot,\cdot)=c_\BEV(\cdot,\cdot)$ (right).}
  \label{fig:smooth}
\end{figure}

For the first example we take the L-shaped domain $\Omega$
sketched in Figure~\ref{fig:Lshape}
and consider the transmission problem \eqref{tp} with prescribed smooth solution
\begin{align*}
  u(x,y) = \frac{x^2+y^2}2, \qquad u^c(x,y) = 0.
\end{align*}
The corresponding data are $f = -2$, $u_0=u|_\Gamma$, $\phi_0 = \partial_{\nn_\Omega} u$.
The definition of the trace norms $\norm\cdot{H^{1/2}(\cS)}$, $\norm\cdot{H^{-1/2}(\cS)}$, 
the quasi-optimality of the methods and standard approximation theory, cf.~\cite{BBF_2013,DemkowiczG_11_ADM},
yield the a priori estimate
\begin{align*}
  \norm{\uu-\uu_{hp}}U \lesssim \inf_{\ww_{hp}\in U} \norm{\uu-\ww_{hp}}U \lesssim 
  h^s (\norm{u}{H^{1+s}(\Omega)}+\norm{f}{H^{s}(\Omega)}), \quad s\le 1.
\end{align*}
Hence, we expect a convergence order $\OO(h)$. This order is confirmed by Figure~\ref{fig:Lshape}
for the least-squares method (right plot), and for both one-equation coupling methods with
Galerkin boundary elements (Figure~\ref{fig:smooth}).
Here, for the least-squares BEM part, we have used the multilevel inner product $\ip\cdot\cdot_\mathrm{ML}$.

\subsection{Example with singular solution}

\begin{figure}
  \begin{center}

    \psfrag{number of elements}[c][c]{\tiny number of elements $N$}
    \psfrag{error and estimator}[c][c]{\tiny error and estimator}
    \psfrag{errU}{\tiny $\norm{u-u_{hp}}{L_2(\Omega)}$}
    \psfrag{errSigma}{\tiny $\norm{\ssigma-\ssigma_{hp}}{L_2(\Omega)}$}
    \psfrag{errEst}{\tiny $\est(\uu_{hp})$}
    \psfrag{Nm13}{\tiny $N^{-1/3}$}
    \psfrag{Nm512}{\tiny $N^{-5/12}$}
    \psfrag{Nm12}{\tiny $N^{-1/2}$}

    \includegraphics[width=0.45\textwidth]{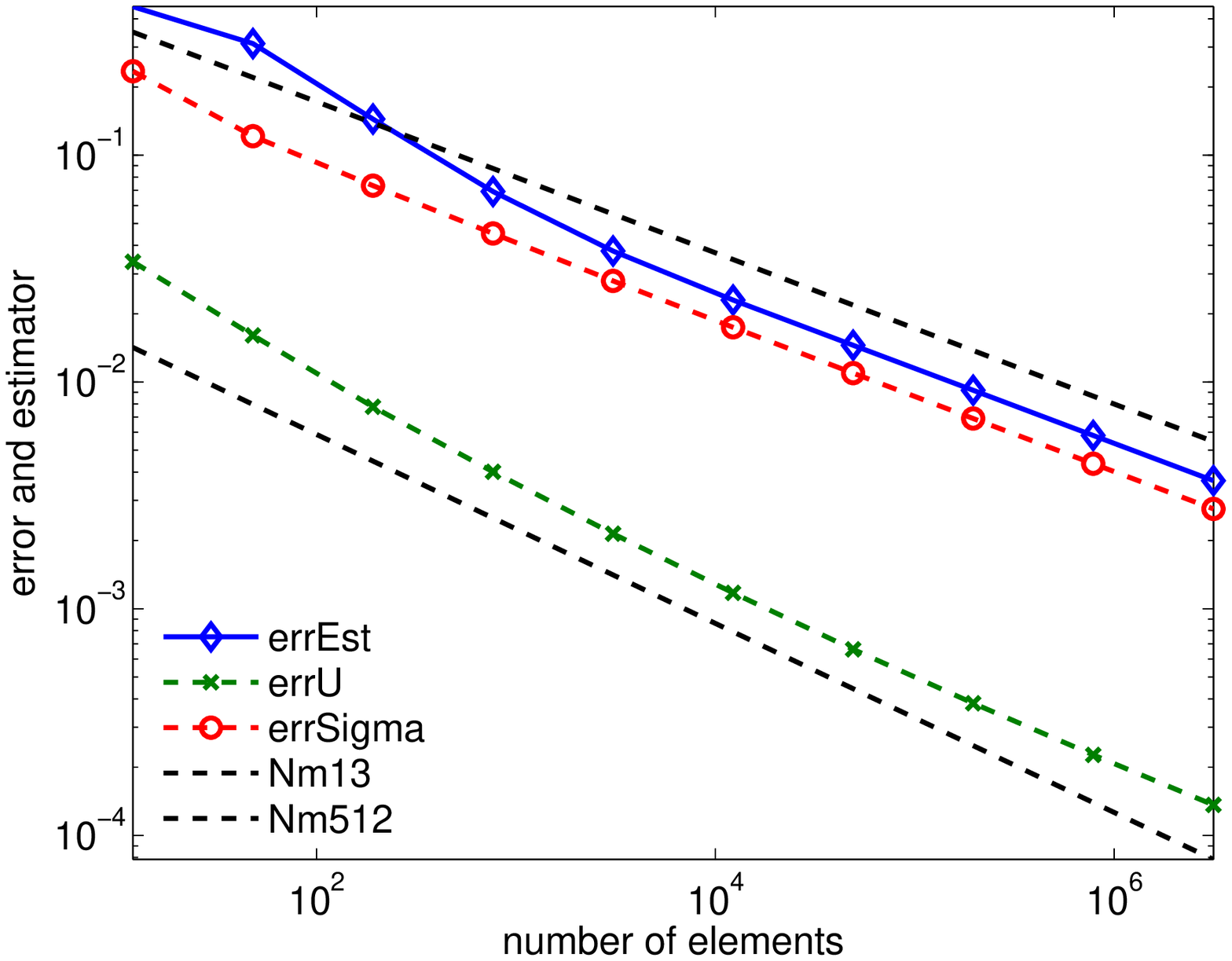}
    \quad
    \includegraphics[width=0.45\textwidth]{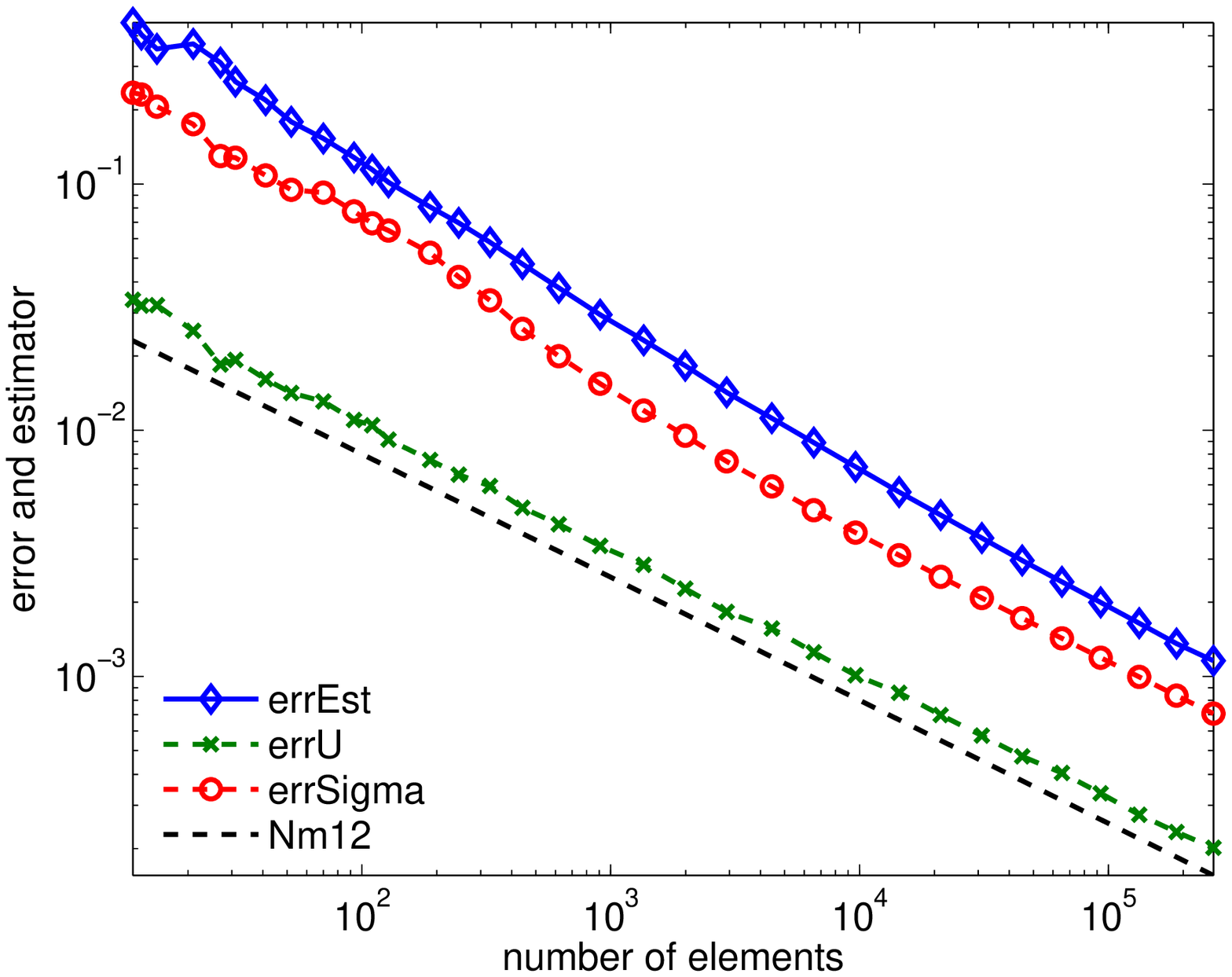}
  \end{center}
  \caption{$L_2$ errors of $u$ and $\ssigma$, and estimator est$(\cdot)$ for least-squares coupling
           and singular solution: uniform meshes (left) and adaptively refined meshes (right).}
  \label{fig:singularLS}
\end{figure}

\begin{figure}
  \begin{center}

    \psfrag{number of elements}[c][c]{\tiny number of elements $N$}
    \psfrag{error and estimator}[c][c]{\tiny error and estimator}
    \psfrag{errU}{\tiny $\norm{u-u_{hp}}{L_2(\Omega)}$}
    \psfrag{errSigma}{\tiny $\norm{\ssigma-\ssigma_{hp}}{L_2(\Omega)}$}
    \psfrag{errEst}{\tiny $\est(\uu_{hp})$}
    \psfrag{Nm13}{\tiny $N^{-1/3}$}
    \psfrag{Nm512}{\tiny $N^{-5/12}$}
    \psfrag{Nm12}{\tiny $N^{-1/2}$}

    \includegraphics[width=0.45\textwidth]{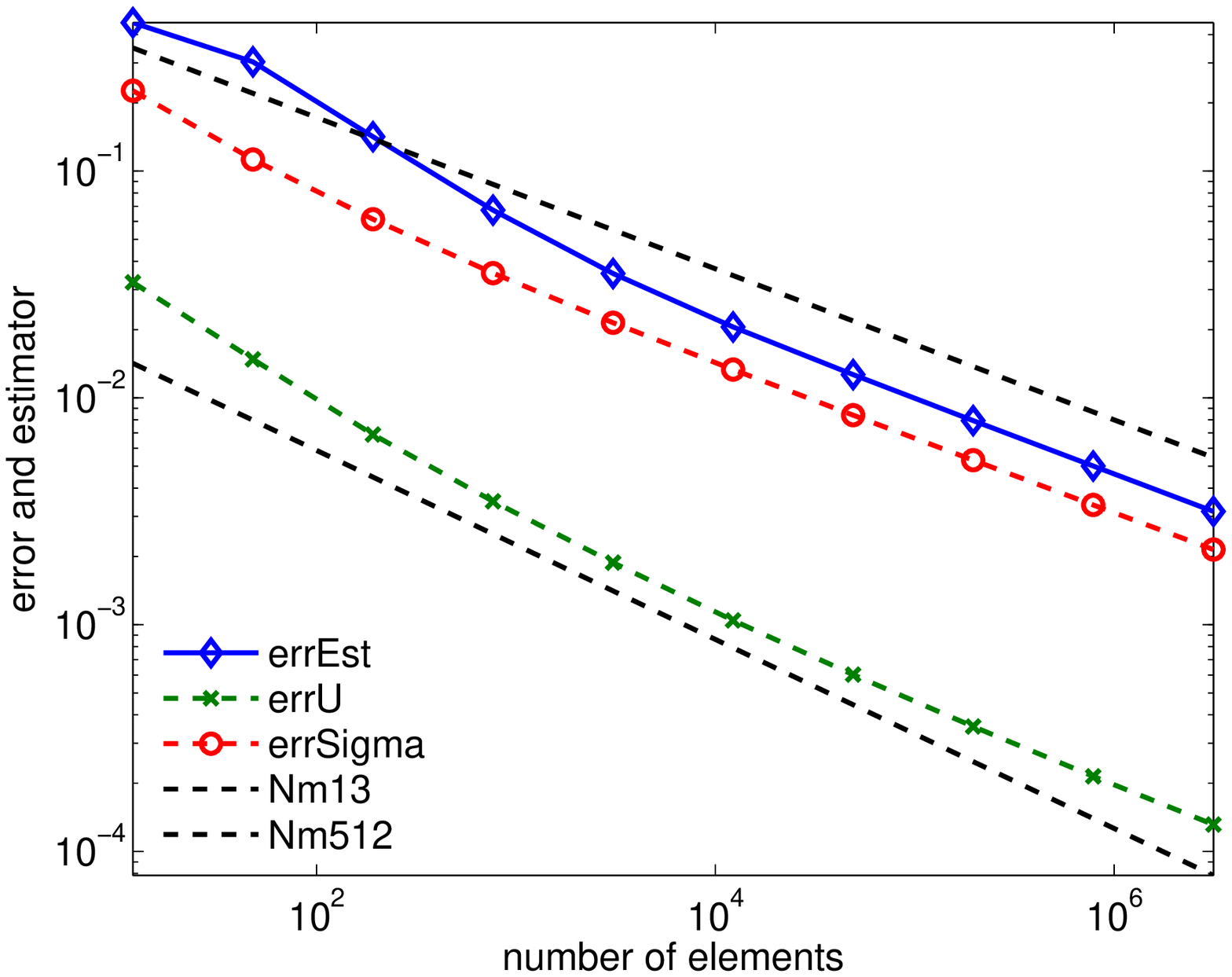}
    \quad
    \includegraphics[width=0.45\textwidth]{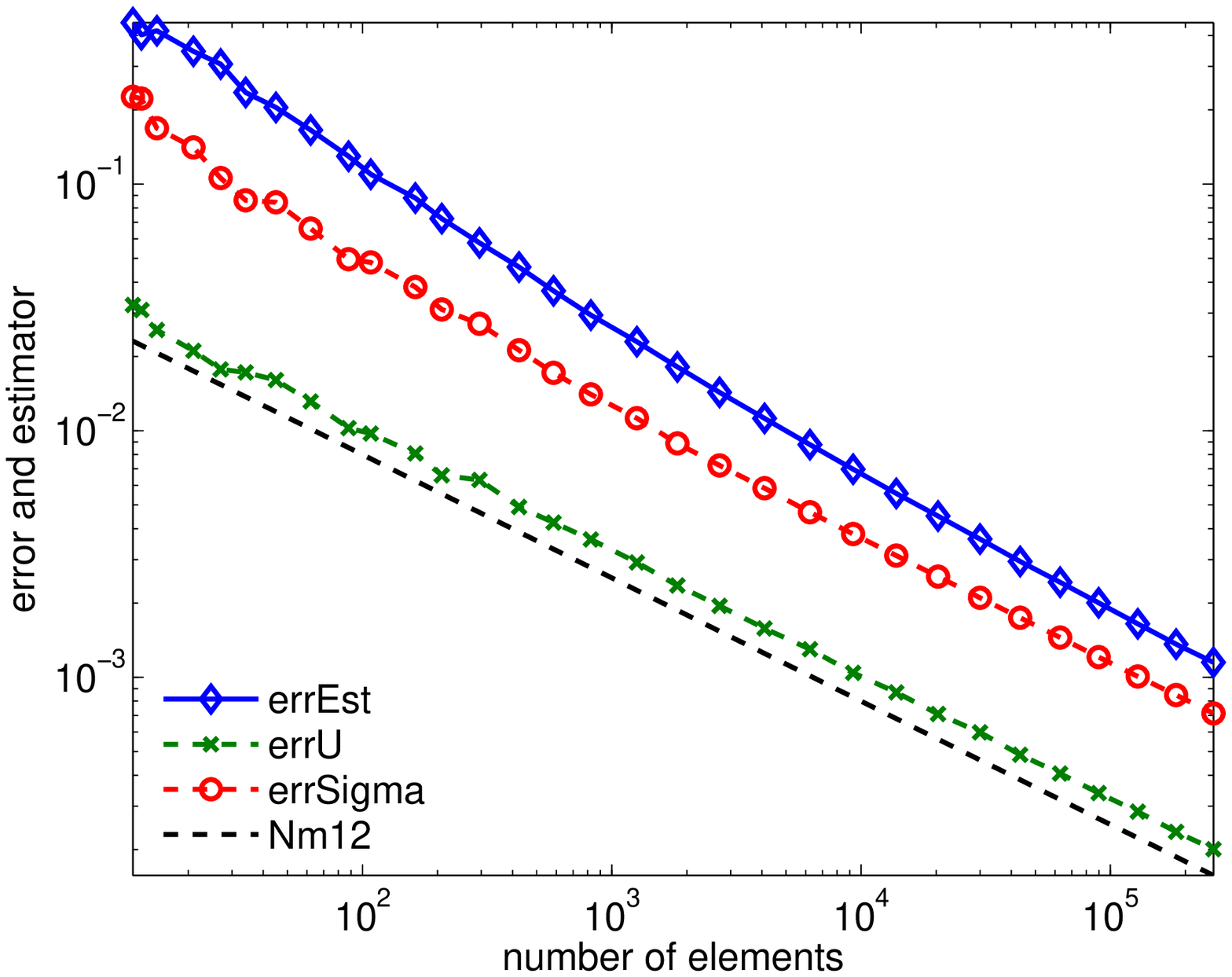}
  \end{center}
  \caption{$L_2$ errors of $u$ and $\ssigma$, and estimator est$(\cdot)$ for coupling with Galerkin boundary
           elements and $c(\cdot,\cdot) = c_\BEW(\cdot,\cdot)$ with singular solution:
           uniform meshes (left) and adaptively refined meshes (right).}
  \label{fig:singularHYP}
\end{figure}

\begin{figure}
  \begin{center}

    \psfrag{number of elements}[c][c]{\tiny number of elements $N$}
    \psfrag{error and estimator}[c][c]{\tiny error and estimator}
    \psfrag{errU}{\tiny $\norm{u-u_{hp}}{L_2(\Omega)}$}
    \psfrag{errSigma}{\tiny $\norm{\ssigma-\ssigma_{hp}}{L_2(\Omega)}$}
    \psfrag{errEst}{\tiny $\est(\uu_{hp})$}
    \psfrag{Nm13}{\tiny $N^{-1/3}$}
    \psfrag{Nm512}{\tiny $N^{-5/12}$}
    \psfrag{Nm12}{\tiny $N^{-1/2}$}

    \includegraphics[width=0.45\textwidth]{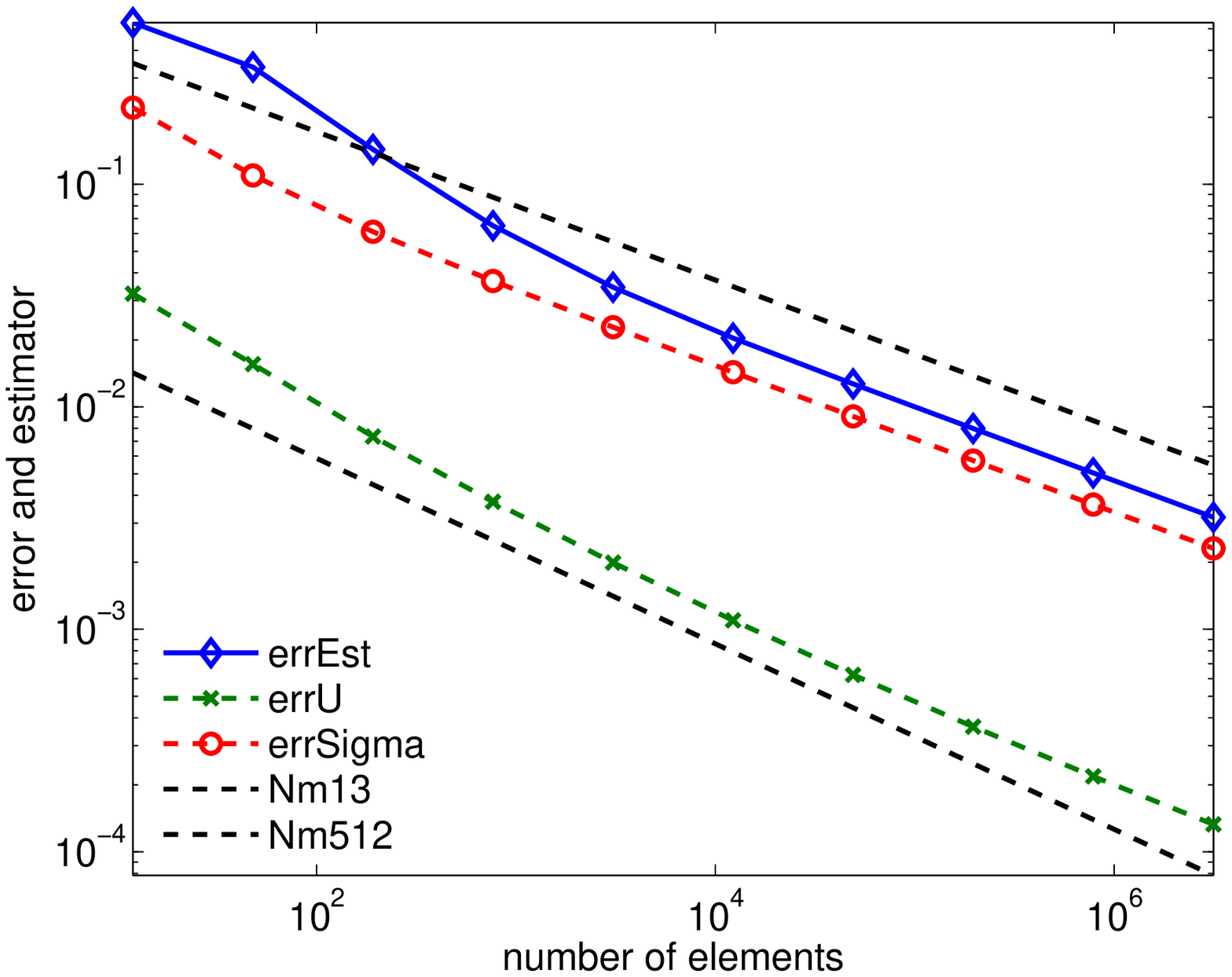}
    \quad
    \includegraphics[width=0.45\textwidth]{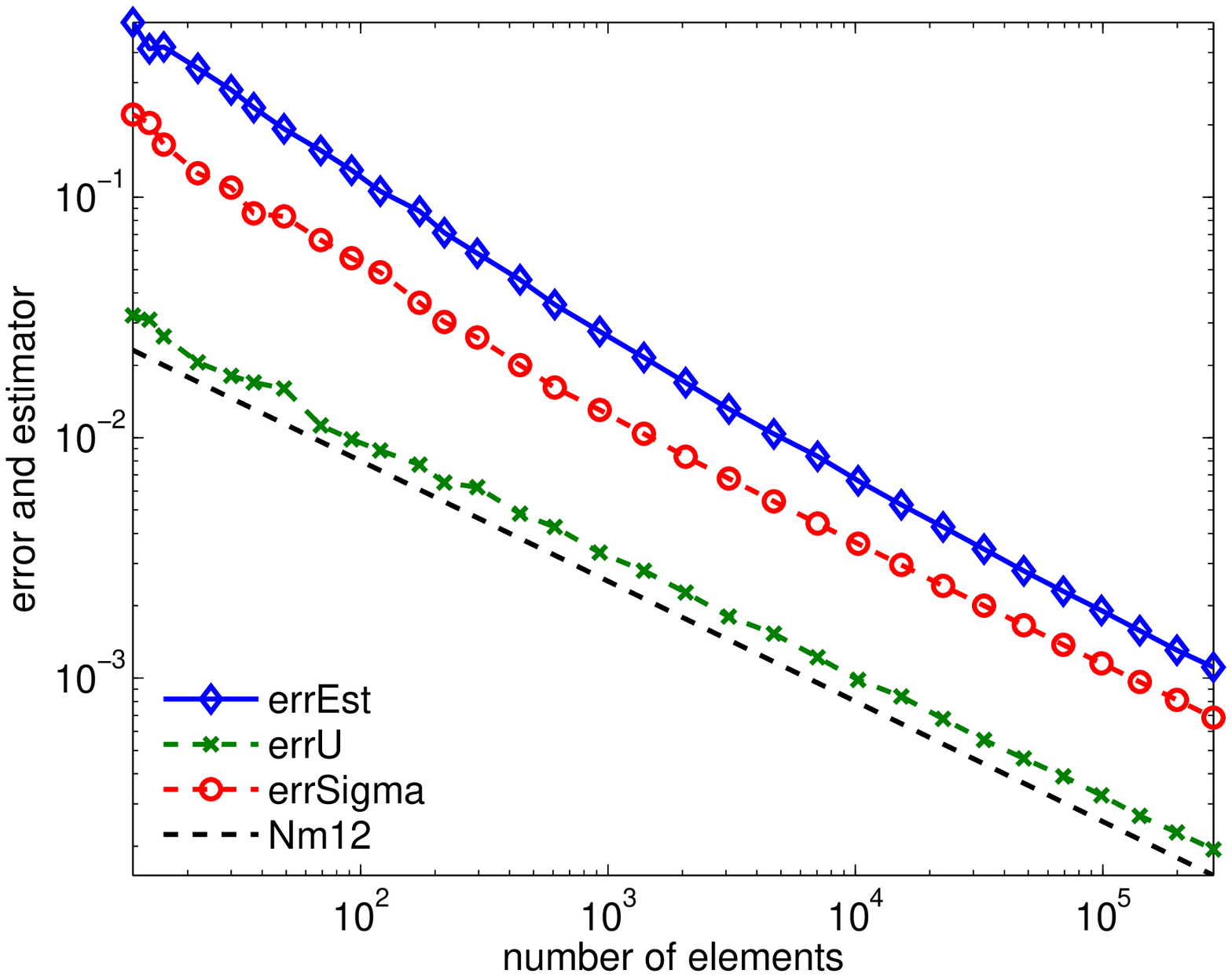}
  \end{center}
  \caption{$L_2$ errors of $u$ and $\ssigma$, and estimator est$(\cdot)$ for coupling with Galerkin boundary
           elements and $c(\cdot,\cdot) = c_\BEV(\cdot,\cdot)$ with singular solution:
           uniform meshes (left) and adaptively refined meshes (right).}
  \label{fig:singularSL}
\end{figure}

We consider the L-shaped domain from Figure~\ref{fig:Lshape} and prescribe the exact solution
\begin{align*}
  u(x,y) &= r^{2/3}\cos(2/3\theta), \quad
  u^c(x,y) = \frac{1}{10} \cdot \frac{x+y-\tfrac18}{(x-\tfrac18)^2 + y^2}.
\end{align*}
Here, $(r,\theta)$ are polar coordinates centered at the incoming corner.
There holds $u\in H^s(\Omega)$ for all $s<5/3$ and, therefore, we expect an overall convergence rate of
$\OO(h^{2/3}) = \OO(N^{-1/3})$. This is confirmed for all coupling methods by the left plots of
Figures~\ref{fig:singularLS},~\ref{fig:singularHYP},~\ref{fig:singularSL}.
We observe that the $L_2$ errors of $u$ converge with a better rate between $\OO(N^{-5/12})$ and $\OO(N^{-1/3})$.

For this example we used $\ip\cdot\cdot_\hyp$ as inner product in $H^{1/2}(\Gamma)$.
This is the simplest selection that also works for non-uniform meshes, needed when considering
an adaptive refinement strategy. For such a strategy we use the quantities
\begin{align*}
  \norm{L-B\uu}{V'}^2 + \norm{h^{1/2} \nabla_\Gamma\V( (u_{0h},\phi_{0h})-\gamma\uu_{hp})}{L_2(\Gamma)}^2
\end{align*}
as indicators, and refine with bulk-strategy (also called Doerfler-marking)
and 30 percent marking (of the squared quantities).
This error estimator is purely heuristical and is used for its locality.
 From Figures~\ref{fig:singularLS},~\ref{fig:singularHYP}, and~\ref{fig:singularSL} (right
plots) we observe that the rate $\OO(N^{-1/2})$ is recovered for all coupling methods.

\bibliographystyle{abbrv}
\bibliography{bib}

\begin{thebibliography}{10}

\bibitem{AuradaEFFFGKMP_HMI}
M.~Aurada, M.~Ebner, M.~Feischl, S.~Ferraz-Leite, T.~F\"uhrer, O.~Goldenits,
  M.~Karkulik, and D.~Praetorius.
\newblock {HILBERT} -- a {MATLAB} implementation of adaptive {2D}-{BEM}.
\newblock {\em Numer. Algorithms}.
\newblock Published online.

\bibitem{AuradaFFKMP_13_CFB}
M.~Aurada, M.~Feischl, T.~F\"uhrer, M.~Karkulik, J.~M. Melenk, and
  D.~Praetorius.
\newblock Classical {FEM}-{BEM} coupling methods: nonlinearities,
  well-posedness, and adaptivity.
\newblock {\em Comp. Mech.}, 51(4):399--419, 2013.

\bibitem{BielakM_91_SFE}
J.~Bielak and R.~C. MacCamy.
\newblock Symmetric finite element and boundary integral coupling methods for
  fluid-solid interaction.
\newblock {\em Quart. Appl. Math.}, 49(1):107--119, 1991.

\bibitem{BochevG_09_LSF}
P.~B. Bochev and M.~D. Gunzburger.
\newblock {\em Least-squares finite element methods}, volume 166 of {\em
  Applied Mathematical Sciences}.
\newblock Springer, New York, 2009.

\bibitem{BBF_2013}
D.~Boffi, F.~Brezzi, and M.~Fortin.
\newblock {\em Mixed finite element methods and applications}, volume~44 of
  {\em Springer Series in Computational Mathematics}.
\newblock Springer, Heidelberg, 2013.

\bibitem{BroersenS_14_RPG}
D.~Broersen and R.~Stevenson.
\newblock A robust {Petrov}-{Galerkin} discretisation of convection-diffusion
  equations.
\newblock {\em Comput. Math. Appl.}, 68(11):1605--1618, 2014.

\bibitem{BroersenS_15_PGD}
D.~Broersen and R.~Stevenson.
\newblock A {Petrov}-{Galerkin} discretization with optimal test space of a
  mild-weak formulation of convection-diffusion equations in mixed form.
\newblock {\em IMA J. Numer. Anal.}, 35(1):39--73, 2015.

\bibitem{ChanHBTD_14_RDM}
J.~Chan, N.~Heuer, T.~Bui-Thanh, and L.~Demkowicz.
\newblock Robust {DPG} method for convection-dominated diffusion problems {II}:
  {Adjoint} boundary conditions and mesh-dependent test norms.
\newblock {\em Comput. Math. Appl.}, 67(4):771--795, 2014.

\bibitem{Costabel_88_BIO}
M.~Costabel.
\newblock Boundary integral operators on {Lipschitz} domains: Elementary
  results.
\newblock {\em SIAM J. Math. Anal.}, 19:613--626, 1988.

\bibitem{Costabel_88_SMC}
M.~Costabel.
\newblock A symmetric method for the coupling of finite elements and boundary
  elements.
\newblock In J.~R. Whiteman, editor, {\em The Mathematics of Finite Elements
  and Applications VI}, pages 281--288, London, 1988. Academic Press.

\bibitem{CostabelS_85_DBI}
M.~Costabel and E.~P. Stephan.
\newblock A direct boundary integral equation method for transmission problems.
\newblock {\em J. Math. Anal. Appl.}, 106:367--413, 1985.

\bibitem{CostabelS_88_CFE}
M.~Costabel and E.~P. Stephan.
\newblock Coupling of finite elements and boundary elements for inhomogeneous
  transmission problems in {$\R^3$}.
\newblock In J.~R. Whiteman, editor, {\em The Mathematics of Finite Elements
  and Applications VI}, pages 289--296, London, 1988. Academic Press.

\bibitem{DemkowiczG_11_ADM}
L.~Demkowicz and J.~Gopalakrishnan.
\newblock Analysis of the {DPG} method for the {Poisson} problem.
\newblock {\em SIAM J. Numer. Anal.}, 49(5):1788--1809, 2011.

\bibitem{DemkowiczG_11_CDP}
L.~Demkowicz and J.~Gopalakrishnan.
\newblock A class of discontinuous {Petrov-Galerkin} methods. {Part II}:
  {O}ptimal test functions.
\newblock {\em Numer. Methods Partial Differential Eq.}, 27:70--105, 2011.

\bibitem{DemkowiczGMZ_12_WEA}
L.~Demkowicz, J.~Gopalakrishnan, I.~Muga, and J.~Zitelli.
\newblock Wavenumber explicit analysis of a {DPG} method for the
  multidimensional {H}elmholtz equation.
\newblock {\em Comput. Methods Appl. Mech. Engrg.}, 213/216:126--138, 2012.

\bibitem{DemkowiczH_13_RDM}
L.~Demkowicz and N.~Heuer.
\newblock Robust {DPG} method for convection-dominated diffusion problems.
\newblock {\em SIAM J. Numer. Anal.}, 51(5):2514--2537, 2013.

\bibitem{ErnGuermond2014}
A.~Ern and J.-L. Guermond.
\newblock {\em Theory and practice of finite elements}, volume 159 of {\em
  Applied Mathematical Sciences}.
\newblock Springer-Verlag, New York, 2004.

\bibitem{arcme}
M.~Feischl, T.~F{\"u}hrer, N.~Heuer, M.~Karkulik, and D.~Praetorius.
\newblock Adaptive {B}oundary {E}lement {M}ethods.
\newblock {\em Arch. Comput. Methods Eng.}, 22(3):309--389, 2015.

\bibitem{dissTOFU}
T.~F{\"u}hrer.
\newblock {\em Zur Kopplung von finiten Elementen und Randelementen}.
\newblock PhD thesis, Vienna University of Technology, 2014.

\bibitem{GaticaHS_03_LSC}
G.~N. Gatica, H.~Harbrecht, and R.~Schneider.
\newblock Least squares methods for the coupling of {FEM} and {BEM}.
\newblock {\em SIAM J. Numer. Anal.}, 41(5):1974--1995 (electronic), 2003.

\bibitem{GaticaHS_12_RHB}
G.~N. Gatica, G.~C. Hsiao, and F.-J. Sayas.
\newblock Relaxing the hypotheses of {B}ielak-{M}ac{C}amy's {BEM}-{FEM}
  coupling.
\newblock {\em Numer. Math.}, 120(3):465--487, 2012.

\bibitem{GopalakrishnanMO_14_DDE}
J.~Gopalakrishnan, I.~Muga, and N.~Olivares.
\newblock Dispersive and dissipative errors in the {DPG} method with scaled
  norms for {H}elmholtz equation.
\newblock {\em SIAM J. Sci. Comput.}, 36(1):A20--A39, 2014.

\bibitem{GopalakrishnanQ_14_APD}
J.~Gopalakrishnan and W.~Qiu.
\newblock An analysis of the practical {DPG} method.
\newblock {\em Math. Comp.}, 83(286):537--552, 2014.

\bibitem{HeuerK_DPG}
N.~Heuer and M.~Karkulik.
\newblock Discontinuous {Petrov}-{Galerkin} boundary elements.
\newblock {http://arXiv.org/abs/1408.5374}, 2014.

\bibitem{HeuerK_DPGb}
N.~Heuer and M.~Karkulik.
\newblock {DPG} method with optimal test functions for a transmission problem.
\newblock {http://arXiv.org/abs/1412.4753}, 2014.
\newblock Accepted for publication in {\em Comput. Math. Appl.}

\bibitem{HeuerP_14_UFH}
N.~Heuer and F.~Pinochet.
\newblock Ultra-weak formulation of a hypersingular integral equation on
  polygons and {DPG} method with optimal test functions.
\newblock {\em SIAM J. Numer. Anal.}, 52(6):2703--2721, 2014.

\bibitem{HsiaoW_08_BIE}
G.~C. Hsiao and W.~L. Wendland.
\newblock {\em Boundary Integral Equations}.
\newblock Springer, 2008.

\bibitem{JohnsonN_80_CBI}
C.~Johnson and J.-C. N\'ed\'elec.
\newblock On the coupling of boundary integral and finite element methods.
\newblock {\em Math. Comp.}, 35:1063--1079, 1980.

\bibitem{kpp}
M.~Karkulik, D.~Pavlicek, and D.~Praetorius.
\newblock On 2{D} newest vertex bisection: optimality of mesh-closure and
  {$H^1$}-stability of {$L_2$}-projection.
\newblock {\em Constr. Approx.}, 38(2):213--234, 2013.

\bibitem{MaischakOS_12_LSF}
M.~Maischak, S.~Oestmann, and E.~P. Stephan.
\newblock A least-squares {FEM}-{BEM} coupling method for linear elasticity.
\newblock {\em Appl. Numer. Math.}, 62(4):457--472, 2012.

\bibitem{MaischakS_04_LSC}
M.~Maischak and E.~P. Stephan.
\newblock A least squares coupling method with finite elements and boundary
  elements for transmission problems.
\newblock {\em Comput. Math. Appl.}, 48(7-8):995--1016, 2004.

\bibitem{McLean_00_SES}
W.~McLean.
\newblock {\em Strongly Elliptic Systems and Boundary Integral Equations}.
\newblock Cambridge University Press, 2000.

\bibitem{os98}
P.~Oswald.
\newblock Multilevel norms for {$H^{-1/2}$}.
\newblock {\em Computing}, 61(3):235--255, 1998.

\bibitem{Sayas_09_VJN}
F.-J. Sayas.
\newblock The validity of {J}ohnson-{N}\'ed\'elec's {BEM}-{FEM} coupling on
  polygonal interfaces.
\newblock {\em SIAM J. Numer. Anal.}, 47(5):3451--3463, 2009.

\bibitem{Steinbach_11_NSO}
O.~Steinbach.
\newblock A note on the stable one-equation coupling of finite and boundary
  elements.
\newblock {\em SIAM J. Numer. Anal.}, 49:1521--1531, 2011.

\bibitem{ZitelliMDGPC_11_CDP}
J.~Zitelli, I.~Muga, L.~Demkowicz, J.~Gopalakrishnan, D.~Pardo, and V.~M. Calo.
\newblock A class of discontinuous {P}etrov-{G}alerkin methods. {P}art {IV}:
  the optimal test norm and time-harmonic wave propagation in 1{D}.
\newblock {\em J. Comput. Phys.}, 230(7):2406--2432, 2011.

\end{thebibliography}
\end{document}